\documentclass[a4paper]{amsart}
\usepackage{amsmath,amsthm,amssymb,latexsym,epic,bbm,comment,mathrsfs}
\usepackage{graphicx,enumerate,stmaryrd, xcolor, color}
\usepackage[all,2cell]{xy}
\xyoption{2cell}
\usepackage{indentfirst} 
\theoremstyle{plain}
\newtheorem{thm}{Theorem}
\newtheorem*{thm*}{Theorem}

\newtheorem{lem}[thm]{Lemma}
\newtheorem{prop}[thm]{Proposition}

\newtheorem{cor}[thm]{Corollary}
\newtheorem{df-prop}[thm]{Definition-Proposition}

\theoremstyle{definition}

\theoremstyle{remark}
\newtheorem{rem}[thm]{Remark}

\usepackage[all]{xy}
\usepackage[active]{srcltx}
\usepackage[parfill]{parskip}
\usepackage{enumerate}

\newcommand{\Hom}{\operatorname{Hom}}

\usepackage{hyperref}
\newcommand{\mc}{\mathcal}
\newcommand{\mf}{\mathfrak}
\newcommand{\C}{\mathbb C}

\newcommand{\oa}{{\bar 0}}
\newcommand{\ob}{{\bar 1}}
\newcommand{\vare}{\epsilon} 
\def\gl{\mathfrak{gl}}
\newcommand{\g}{\mathfrak{g}}

\def\la{\lambda}
\def\pn{\mf{pe} (n)}
\def\ov{\overline}
\newcommand{\ch}{\mathrm{ch}}
\newcommand{\h}{\mathfrak{h}}

\newcommand{\Z}{{\mathbb Z}}

\def\mod{\operatorname{-mod}\nolimits}

\def\Hom{\operatorname{Hom}\nolimits}
\def\End{\operatorname{End}\nolimits}
\def\Res{\operatorname{Res}\nolimits}
\def\Ind{\operatorname{Ind}\nolimits}
\def\Ext{\operatorname{Ext}\limits}

\def\pr{\operatorname{pr}\nolimits}

\def\gl{\mathfrak{gl}}

\def\la{\lambda}

\def\pn{\mf{pe} (n)}
\def\ov{\overline}

\newcommand{\wtg}{\widetilde{\g}}
\newcommand{\wto}{\widetilde{\mc O}}
\newcommand{\wtV}{\widetilde{M}}
\newcommand{\wtL}{\widetilde{L}}
\newcommand{\Id}{\text{Id}}

\newcommand{\Tor}{\text{Tor}}
\def\Mod{\operatorname{-Mod}\nolimits}

\newcommand{\tens}[1]{%
	\mathbin{\mathop{\otimes}\limits_{#1}}%
}

\newcommand{\hellgeq}{\h^{\scriptscriptstyle  \geq 1}_\ell}

\begin{document}
\title[Typical 
representations of Takiff superalgebras]{Typical 
representations of Takiff superalgebras}

\author[Chen]{Chih-Whi Chen} \address{Department of Mathematics, National Central University, Chung-Li, Taiwan 32054 \\ National Center of Theoretical Sciences,
	Taipei, Taiwan 10617} \email{cwchen@math.ncu.edu.tw}
\author[Wang]{Yongjie Wang} \address{School of Mathematics, Hefei University of Technology, Hefei, Anhui, 230009, China} \email{	wyjie@mail.ustc.edu.cn}
\begin{abstract} 
 We investigate representations of the $\ell$-th Takiff superalgebras $\wtg_\ell := \wtg\otimes \C[\theta]/(\theta^{\ell+1})$, for $\ell>0$, associated with  a basic classical and a periplectic Lie superalgebras  $\wtg$. We introduce the odd reflections and  formulate a general notion of  typical representations of the Takiff superalgebras $\wtg_\ell$. As a consequence, we provide a complete description of the characters of the finite-dimensional modules over type I Takiff superalgebras.  For the Lie superalgebras $\wtg= \gl(m|n)$ and  $\mf{osp}(2|2n)$, we prove that the Kac induction functor of $\wtg_\ell$ leads to  an equivalence from an arbitrary typical Jordan block of the  category $\mc O$ for $\wtg_\ell$ to a Jordan block of the category  $\mc O$ for the even subalgebra of $\wtg_\ell$.
We also obtain a  classification of non-singular simple Whittaker modules over the Takiff superalgebras.
\end{abstract}

\maketitle

\noindent
\textbf{MSC 2010:} 17B10, 17B55

\noindent
\textbf{Keywords:} Takiff  algebra; Takiff  superalgebra; Category $\mc O$; Typical representation; Whittaker module.

\vspace{5mm}

\section{Introduction}\label{sec1}
\subsection{Takiff algebras}
Let $\g$ be a complex finite-dimensional Lie algebra and $\ell$ be a non-negative integer.  The corresponding {\em $\ell$-th Takiff   algebra},\footnote{This
 terminology is more general than other usage of the same	term elsewhere in the literature, where "Takiff algebra" is typically used to indicate   the $1$-st Takiff algebra in the sense of \mbox{Takiff in \cite{T71}}.} as introduced by Takiff \cite{T71} in the case when $\ell=1$, is a Lie algebra over a truncated polynomial algebra:
\begin{align*}
&\g_\ell : = \g\otimes \C[\theta]/(\theta^{\ell+1}).
\end{align*}   
Such a Lie algebra was also considered in the literature under the name {\em truncated current Lie algebra} or {\em polynomial Lie algebra}, and it   appeared in numerous branches of mathematics and mathematical physics. This has led to  significant applications in various fields; see also \cite{CO06,  Wi11,BR13, MY16, AP17, MacS19, PY20, M21}. 

There has also been considerable recent progress in the representation theory of Takiff algebras associated with complex semisimple Lie algebras.  In \cite{MS19}, Mazorchuk and S\"oderberg  initiated the study of the category $\mc O$ for the $1$-st Takiff algebra  associated with $\mf{sl}(2)$. In loc. cit., the authors  described the  Gabriel quivers for blocks  and  determined the composition multiplicities of simple modules inside Verma modules for the $1$-st Takiff $\mf{sl}(2)_1$. For the more general $\ell$-th Takiff algebras associated with reductive Lie algebras $\g$, the problem of composition multiplicities of Verma modules has been solved by recent works of Chaffe and Topley  \cite{Ch23, CT23}.  The answer
	is formulated in terms of the Kazhdan-Lusztig combinatorics of the category $\mc O$ of $\g$. 
    
 To explain this in more detail, we fix  a triangular decomposition $\g = \mf n\oplus \mf h\oplus \mf n^-$  of a reductive Lie algebra $\g$. For a given subalgebra $\mf s\subseteq \g$, we define $\mf s_\ell:=\bigoplus_{i=0}^{\ell}\mf s \otimes \mathbb C\theta^i.$ This leads to a {\em triangular decomposition} of   $\g_\ell$  in the sense of \cite{CT23} (see also \cite{MS19, Ch23}):
\begin{align*}
&\g_\ell =\mf n_\ell\oplus \mf h_\ell \oplus \mf n^-_\ell.
\end{align*}  Let $\mf b_\ell :=\mf n_\ell \oplus \mf h_\ell$. We identify  $\g$  as the subalgebra  $\g\otimes 1=\g\otimes \theta^0$. The category $\mc O_\ell$ consists of all  finitely generated $\g_\ell$-modules on which  $\h$ acts semisimply and $\mf b_\ell$ acts locally finitely. For   $\la \in \h^\ast_\ell$,  let $\C_\la$ be the one-dimensional ${\mf h_\ell}$-module induced by $\la$. We may extend $\C_\la$ to be a ${\mf b}_\ell$-module by letting $\mf n_\ell\cdot \C_\la =0$.  The corresponding  Verma module  is defined as $$M(\la):= U(\g_\ell)\otimes_{U(\mf b_\ell)} \C_\la. $$
 The tops of $M(\la)$ ($\la \in \h^\ast_\ell$), which we shall denote by $L(\la)$, exhaust all simple highest weight modules over $\g_\ell$, that is, all simple objects in $\mc O_\ell$. Let $\hellgeq:=\bigoplus_{i=1}^\ell \h\otimes \theta^i$. Then the category $\mc O_\ell$ decomposes as  $\mc O_\ell = \bigoplus_{\kappa\in (\hellgeq)^\ast}\mc O_\ell^{\kappa},$      with  the full subcategory  $\mc O_\ell^\kappa$  consisting of modules $M\in \mc O_\ell$ on which $x-\kappa(x)$ acts locally nilpotently, for each $x\in \hellgeq$.  Following \cite{CT23}, the subcategories $\mc O^{\kappa}_\ell$ are called the {\em Jordan blocks} of $\mc O_\ell$.  
 Using the parabolic induction   and twisting functors,  there has been developed in \cite{CT23, Ch23}  a reduction procedure which provides an equivalence from an arbitrary Jordan block  $\mc O_\ell^\kappa$ to a Jordan block of $\mc O_\ell$ for a certain Takiff  subalgebra $\mf l_\ell$ associated with a Levi subalgebra $\mf l$ of $\g$. 
We refer to \cite{MS19, Ch23, CT23} for all the details mentioned above. We also refer to \cite{He18, He22, MM22, X23, Z24}  and references
therein for more results on the representation theory of Takiff algebras.

\subsection{Takiff superalgebras}
Fix a finite-dimensional Lie superalgebra $\wtg =\wtg_\oa\oplus \wtg_\ob$, see, for example, \cite{Ka1}. We  also refer to \cite{CW12}  for more details about the representation theory of Lie superalgebras. Throughout the paper, we set $$\g:=\wtg_\oa.$$ The constructions of Takiff algebras afford a natural superalgebra generalization, namely, for any non-negative integer $\ell\geqslant0$, we  consider the {\em $\ell$-th Takiff superalgebra} associated with \mbox{$\wtg$}: 
\begin{align*}
	&\wtg_\ell:= \wtg\otimes \C[\theta]/(\theta^{\ell+1}),
\end{align*} where $\theta$ is an even element. 
Such a Lie superalgebra and its variations have appeared many times in the mathematics and physics literature (e.g., see,  \cite{BR13,Sa14,BC15, GM17,  Q20, ChCo22, CCS24}).

To explain the main results of the paper in more detail, we start by explaining
our   setup. Throughout the paper,  we assume that $$\ell>0,$$ and    $\wtg$ is  a basic classical or a periplectic Lie superalgebra in the sense of Kac \cite{Ka1}; see the list \eqref{eq::Kaclist}. In particular, $\wtg$ is a finite-dimensional quasi-reductive  Lie superalgebra in the sense of \cite{Se11}. This means that $\g$ ($=\wtg_\oa$)
is a reductive Lie algebra, and $\wtg_\ob$
is semisimple as a $\g$-module. Such a Lie superalgebra admits the notion of triangular decompositions;  see also \cite{Se11, Ma14,CCC21}.
In the paper, we will fix the triangular decomposition of $\wtg$ as in \eqref{eq::tria}:
\begin{align}
	&\wtg =\widetilde{\mf n}\oplus \mf h \oplus \widetilde{\mf n}^-, \label{eq::tria}
\end{align} with a purely even Cartan subalgebra $\h$ and and nilradicals $\widetilde{\mf n},~\widetilde{\mf n}^-$ and the corresponding Borel subalgebra $\widetilde{\mf b} = \mf h\oplus \widetilde{\mf n}$.

 \subsubsection{The category $\wto_\ell$} \label{sect::123}
 We shall identify $\wtg$ with the subalgebra $\wtg\otimes 1\subseteq \wtg_\ell$. 
 For any subalgebra $\mf s $  of $\wtg$, we define the subalgebra $\mf s_\ell:= \mf s\otimes \C[\theta]/(\theta^{\ell+1})= \bigoplus_{i=0}^{\ell}\mf s \otimes \mathbb C\theta^i\subseteq \wtg_\ell$.  Then the triangular decomposition from \eqref{eq::tria} induces a  decomposition of $\wtg_\ell$: 
\begin{align}
	&\wtg_\ell = \widetilde{\mf n}_\ell \oplus {\mf h}_\ell \oplus  \widetilde{\mf n}_\ell^-.  \label{eq::4}
\end{align} 
Similar to the Lie algebra module case, we define the category $\wto_\ell$ as the category  consisting of all finitely generated $\wtg_\ell$-modules  on which $\h$ acts semisimply and $\widetilde{\mf b}_\ell$ acts locally finitely. 
For $\la\in \h^\ast_\ell$, we   extend the $\h_\ell$-module $\C_\la$ to a $\widetilde{\mf b}_\ell$-module by letting $\widetilde{\mf n}_\ell\cdot \C_\la=0$. Similarly, we define the Verma module over $\wtg_\ell$ with highest weight $\la$ via $\widetilde{M}(\la) :=U(\wtg_\ell)\otimes_{U(\widetilde{\mf b}_\ell)} \C_\la.$ As in the Lie algebra case, the Verma modules $\widetilde{M}(\la)$ and their simple tops  $\widetilde{L}(\la)$ are parametrized by  elements $\la \in \h^\ast_\ell$, and these $\widetilde{L}(\la)$ exhaust all simple objects in $\wto_\ell$;  see also Subsection \ref{sect::222}. For $\kappa\in (\hellgeq)^\ast$, we let $\widetilde{\mc O}^\kappa_\ell$ be the full subcategory of $\wto_\ell$  consisting of objects that are restricted to $\g_\ell$-modules in $\mc O^\kappa_\ell$. Then  we have a decomposition $\wto_\ell = \bigoplus_{\kappa\in (\hellgeq)^\ast}\wto_\ell^{\kappa}$.  Again, we refer to the subcategories $\wto_\ell^\kappa$ as the {\em Jordan blocks} in $\wto_\ell$.

\subsubsection{Typical weights} \label{sect::124} 
  For each $0\leq i\leq \ell$, we identify  $\h\otimes \theta^i$ with $\h$ via the  natural isomorphism $h\otimes \theta^i \mapsto h$ ($h\in \h$). Hence, we obtain the  identifications $\h_\ell\cong \h^{\oplus \ell+1}$, $\h^\ast_\ell \cong (\h^\ast)^{\oplus \ell+1}$, $\hellgeq\cong \h^{\oplus \ell}$ and $(\hellgeq)^\ast\cong (\h^\ast)^{\oplus \ell}$.  Transferred via the above identification,   for each $\la\in \h^\ast_\ell$, we can set  the following notations:
\begin{align}
	&\la \equiv (\la^{(0)},\la^{(1)},\ldots ,\la^{(\ell)}) \in (\h^\ast)^{\oplus \ell+1}, \label{eq::blockdec}
\end{align} where $\la^{(i)}$ is induced by the restriction of $\la$ to $\h\otimes \theta^i$. We  let $\la^{\scriptscriptstyle (\geq 1)}\in (\hellgeq)^\ast$  denote the restriction of the weight $\la$ to $\hellgeq$. We shall fix a  non-degenerate symmetric bilinear form $\langle\_,\_\rangle$ on $\h^\ast$ in Subsection \ref{sect::2.1}. Let $\ov{\Phi}^+_\ob\subseteq \h^\ast$ be the set of all isotropic positive odd roots.   A weight $\la\in \h^\ast_\ell$ is called  {\em typical}  provided that \begin{align*}
&\prod_{\alpha\in \ov{\Phi}^+_\ob}\langle\la^{(\ell)},\alpha\rangle\neq 0, \text{ if $\wtg$ is basic classical};\\
&\prod_{\alpha\in \Phi_\oa^+}\langle\la^{(\ell)},\alpha\rangle\neq 0, \text{ if $\wtg$ is periplectic}.
\end{align*} This is reminiscent of the definition of typical weights in the sense of \cite{Ka78} and \cite{Se02}; see also Subsection \ref{sect::252}. The formulation of a typical weight for the $1$-st Takiff superalgebra associated with $\gl(1|1)$ originates in the earlier work of Babichenko and Ridout \cite{BR13}.

\subsubsection{Takiff superalgebras of type I}
 A quasi-reductive Lie superalgebra $\wtg$ is said to be a {\em type I Lie superalgebra} if it  possesses a $\Z_2$-compatible  $\Z$-gradation of the form
\begin{align}
	&\wtg = \wtg^{1}\oplus \wtg^0\oplus \wtg^{-1}, \label{eq::typeIgr}
\end{align} namely, $ \wtg^0 = \g$ and  $\wtg_\ob = \wtg^{1}\oplus \wtg^{-1}$. We refer to the grading  \eqref{eq::typeIgr}  as a {\em type-I} grading of $\wtg$. 
We are mainly interested in the following    type I	Lie superalgebras from Kac’s list \cite{Ka1}:
\begin{align}
	&\widetilde{\g} = \gl(m|n), \mf{osp}(2|2n),\text{ or } \pn. \label{eq::typeI}
\end{align}    In this case, the triangular decomposition  are supposed to be {\em distinguished},  that is,  $\widetilde{\mf n}^-_\ob = \wtg^{-1}$ and $\widetilde{\mf n}_\ob = \wtg^{1}$.

\subsubsection{Goals} The goal of the present paper is to study several aspects of the representation theory of the $\ell$-th Takiff superalgebras $\wtg_\ell$   associated with a basic classical Lie superalgebra or a periplectic Lie superalgebra. Namely,  the present paper attempts to 
 formulate the typical representations and develop the  odd reflections of $\wtg_\ell$ to  study the characters of finite-dimensional $\wtg_\ell$-modules and typical simple highest weight $\wtg_\ell$-modules and their applications to Whittaker $\wtg_\ell$-modules.

\subsection{The main results} 
\subsubsection{Typical representations and simplicity of Kac modules}  For a complex Lie superalgebra $\mf s$ we denote the category of    finitely-generated $\mf s$-modules by $\mf s\mod$.  We denote the universal enveloping algebra of $\mf s$ by $U(\mf s)$ and its center by $Z(\mf s)$.

For each $\la\in \h^\ast_\ell$, the center $Z(\g_\ell)$ acts on $L(\la)$ by some character, which we shall denote by $\chi_\la(\_):Z(\g_\ell)\rightarrow\C$ and refer to as {\em the central character of $L(\la)$}; see \cite[Subsection 4.1]{CT23} for more details. Similarly, we define $\widetilde{\chi}_\la:  Z(\wtg_\ell)\rightarrow \C$ to be  the central character of $\wtL(\la)$.  
 
 Let $K(\_):\g_\ell\mod \rightarrow \wtg_\ell\mod$ be the Kac induction functor (see Subsection \ref{sect::24}), for type I Lie superalgebras $\wtg$.   Our first main result is the following formulation of {\em typical central characters} and
criteria for simplicity of Kac modules over type I Lie superalgebras:
\begin{thm}  \label{thm::1st}  Let $\wtg_\ell$ be the $\ell$-th Takiff superalgebra associated with  Lie superalgebra $\wtg$.  
 \hskip0.2cm
\begin{itemize}
        \item[(i)] Suppose that $\wtg$ is a basic classical Lie superalgebra.  Then there is a central element $\widetilde{\Omega}\in Z(\wtg_\ell)$ such that 
        $$\widetilde{\chi}_\la(\widetilde{\Omega}) = \prod_{\alpha}\langle \la^{(\ell)}, \alpha\rangle, \text{ for any }\la\in \h^\ast_\ell,$$ where $\alpha$ takes over all positive odd isotropic roots. We refer to central characters of $\wtg_\ell$ which do not vanish at $\widetilde{\Omega}$ as the typical central characters.
	\item[(ii)] 	Suppose that $L$ is a simple module over $\wtg_\ell$ associated with $\wtg=\gl(m|n)$ or  $\mf{osp}(2|2n)$. Let $\widetilde{\chi}$ be the central character of $L$. 
Then we have 
   \[ \widetilde{\chi} \text{ is typical } \Leftrightarrow L\cong K(V), \text{for some simple $\g_\ell$-module $V$.}\] Moreover,    $\widetilde{\chi}_\la$ is typical if and only if $\la$ is typical,    for any  $\la\in\h^\ast_\ell$. 
	\item[(iii)] Suppose that $\wtg =\pn$. Then, for any weight $\la \in \h^\ast_\ell$, we have  $$K(\la)\text{ is simple}\Leftrightarrow \la \text{ is typical.}$$  
\end{itemize}		
\end{thm}

 One of our main strategies of the proof of Theorem \ref{thm::1st} is to develop an analogue of the odd reflections for Takiff superalgebras, which was  introduced by Leites, Saveliev and Serganova  \cite{LSS86}; see Subsection \ref{sect24::OddRef}. As a consequence, we give in Theorem \ref{thm24::characters} a complete description of irreducible characters of finite-dimensional modules over the type I Takiff superalgebras $\wtg_\ell$ associated with Lie superalgebras $\wtg$ of types $\mf{gl}, \mf{osp}$ or $\mf p$. This extends the earlier work of Babichenko and Ridout in \cite{BR13}, where the simple  modules over the  $1$-st Takiff superalgebra  associated with $\gl(1|1)$ were studied.   

\subsubsection{Equivalences between typical Jordan blocks in $\widetilde{\mc O}_\ell$ and Jordan blocks in $\mc O_\ell$} \label{sect::132}

 A Jordan block   $\wto_\ell^\kappa$  is called  {\em typical} if it contains a simple module $\wtL(\la)\in \wto_\ell$ with a typical  highest weight $\la \in \h^\ast_\ell$. Namely, if we write $\kappa = (\kappa^{(1)},\ldots, \kappa^{(\ell)})\in  (\h^\ast)^{\oplus \ell}$, transferred via the identification  $(\hellgeq)^\ast\cong (\h^\ast)^{\oplus \ell}$ from Subsection \ref{sect::124}, then $\wto_\ell^\kappa$ is  typical  if and only if  $\langle\kappa^{(\ell)},\alpha\rangle\neq 0$, for all isotropic odd roots $\alpha$. 
   Our second main result is the following.  	
 
\begin{thm}  \label{thm::2nd} Suppose that $\wtg=\gl(m|n)$ or $\mf{osp}(2|2n)$. Let  $\wto_\ell^\kappa$ be a typical Jordan block of $\wto_\ell$, for some  $\kappa\in (\hellgeq)^\ast$. Let $X_\kappa\subset \h^\ast_\ell$ be the set of all highest weights of simple objects in $\wto^\kappa_\ell$. Then  the Kac induction functor  $K(\_)$ restricts to an equivalence of categories 
	\begin{align*}
		&K(\_):\mc O_\ell^\kappa \xrightarrow{\cong} \wto_\ell^\kappa,  
	\end{align*} sending  $  M(\la)$ to   $\widetilde M(\la)$ and $L(\la)$ to $\wtL(\la)$, for $\la\in X_\kappa$, respectively.
\end{thm} 

 When combined with \cite[Theorem 1.1, Corollary 6.8]{CT23},  Theorem \ref{thm::2nd}  provides a complete solution to the composition multiplicity problem of Verma modules over $\wtg_\ell$ with typical highest weights,  in terms of the Kazhdan-Lusztig polynomials for $\g$.  
 The equivalence in Theorem \ref{thm::2nd} is reminiscent of a remarkable result of Gorelik \cite{Go02} in the setting of category $\mc O$ for a basic Lie superalgebra $\wtg$. In  loc. cit., Gorelik proved that the so-called strongly typical blocks of the category $\mc O$ for $\wtg$ are
 equivalent to the corresponding blocks of the category $\mc O$ for $\g$.

\subsubsection{Simplicity of the  standard Whittaker modules}

Let   $\g$ be a finite-dimensional semisimple Lie algebra  with a triangular decomposition $\g =\mf n\oplus \mf h\oplus \mf n^-$. In his 1978 seminal paper \cite{Ko78},  Kostant  studied a family of locally $\mf n$-finite  simple  $\g$-modules $Y_{\zeta,\eta}$ containing the so-called {\em Whittaker vectors}
associated with {\em a non-singular character}  $\zeta:\mf n\rightarrow\C$,  namely,  $\zeta$ is a character that does not vanish on any simple root vector of $\g$. In recent years, following the   work of Kostant, there has been a recent surge of considerable interest in Lie algebra representations  that are locally finite over nilpotent subalgebras, which we refer to as  {\em Whittaker modules};  see also \cite{Mc85, Mc93, MiSo97, B97,  BM11, CoM15, BR20, AB21}. A large part of the motivation comes from the  connections between the Whittaker modules over Lie algebras and the representations of  associated finite $W$-algebras and other areas; see, e.g., \cite{Sev00, Lo10, Lo10b, Wa11}.

The most elementary among the Whittaker modules over a reductive Lie algebra $\g$ with respect to the nilradical $\mf n$ is the so-called {\em standard Whittaker
module}, which is doubly parameterized by central characters of $\g$ and characters of $\mf n$.  These modules were  originally  introduced   by  Mili{\v{c}}i{\'c} and Soergel \cite{MiSo97} and McDowell   \cite{Mc85,Mc93}, and subsequently, they have played an important role in various investigations. 
We also refer to \cite{BCW14, C21a,  CC23, CC22, CCM23} for   more recent results on Whittaker modules over quasi-reductive Lie superalgebras and their connections with finite $W$-superalgebras.

While Whittaker modules for reductive Lie algebras are, by now, well-understood, their Takiff analogues  were studied in detail only very recently. In a recent exposition \cite{He18, He22}, He initiated the study of Whittaker modules over the Takiff algebras $\g_\ell$ associated with semisimple Lie algebras $\g$ with respect to the Takiff  subalgebras  $\mf n_\ell$. 
In particular, in loc. cit. the author obtained a class of simple Whittaker $\g_\ell$-modules, associated with certain  non-singular characters of $\mf n_\ell$; see Subsection \ref{sect::42}. Subsequently, it was proved by Xia \cite{X23}   that the following  $\g_\ell$-modules exhaust all simple Whittaker modules (in a slightly different form) associated to non-singular characters $\zeta: \mf n_\ell\rightarrow \C$:
$$L(\chi,\zeta) := {U}(\mathfrak{g}_{\ell})\otimes_{Z(\mathfrak{g}_{\ell})U(\mathfrak{n}_{\ell})}\C_{\chi,\zeta},$$
where $\chi: Z(\g_\ell)\rightarrow\C$ is  a central character and $\C_{\chi,\zeta}$ is the associated one-dimensional $Z(\g_\ell)U(\mf n_\ell)$-module; see Lemma \ref{cor::26}. In light of these results, we  realize the category of Whittaker  $\g_\ell$-modules associated with an arbitrary non-singular character of $\mf n_\ell$ as the category of finite-dimensional modules over $Z(\g_\ell)$, in spirit of Kostant's equivalence \cite{Ko78}; see Proposition \ref{prop::27}.

 The final piece of motivation for the present paper is to study Whittaker modules over the Takiff superalgebras considered in the present paper. 
 We assume that $\wtg_\ell$ is a Takiff superalgebra  associated with  $\wtg = \gl(m|n), \mf{osp}(2|2n),$ or $\pn$ with a distinguished triangular decomposition. Fix a non-singular character  $\zeta$ of the the even subalgebra $\mf n_\ell$ of $\widetilde{\mf n}_\ell$. In the paper, a $\wtg_\ell$-module $M$ is said to be a Whittaker module associated to $\zeta$ provided that $M$ is  finitely generated over $U(\wtg_\ell)$, locally finite over $Z(\g_\ell)$ and $x-\zeta(x)$ acts on $M$ locally nilpotently, for any $x\in \mf n_\ell$. 
Denote by $\widetilde{\mc N}_\ell(\zeta)$ the category of all   Whittaker modules associated to $\zeta$.  In contrast to the category $\mc \wto_\ell$,  we prove that every object in $\widetilde{\mc  N}_\ell(\zeta)$ has finite length; see also Proposition \ref{prop::27}. For a central character $\chi:Z(\g_\ell)\rightarrow\C$, we define the corresponding {\em standard Whittaker module} over $\wtg_\ell$:
$$\widetilde M(\chi,\zeta) : = K(L(\chi,\zeta)).$$

The following is our third main result: 
\begin{thm} \label{thm::3rd} 
Suppose that  $\wtg=\gl(m|n),  \mf{osp}(2|2n)$ or $\pn$.  Let $\zeta:  \mf{n}_\ell\rightarrow \C$ be a non-singular character.  	 Then the isomorphism classes of  simple objects in $\widetilde{\mc N}_\ell(\zeta)$ can be represented  as the simple tops of $\widetilde M(\chi,\zeta)$, for characters $\chi: Z(\g_\ell)\rightarrow\C$.  Furthermore, for  $\wtg=\gl(m|n)$, or $\mf{osp}(2|2n)$ with $\la \in\h^\ast_\ell$, we have  
\begin{align}
&\widetilde M(\chi_\la,\zeta) \text{ is simple } \Leftrightarrow  \la \text{ is typical}.  \label{eq::9}
\end{align}  
\end{thm}

\subsection{Structure of the paper}  The paper is organized as follows. In Subsections \ref{sect::200}, \ref{sect::2.1}, we set up the usual description of the   triangular decompositions of quasi-reductive Lie superalgebras.  In Subsections \ref{sect::22}, we provide some background materials on the category $\wto_\ell$ of the Takiff superalgebras. Subsection \ref{sect24::OddRef} is devoted to   the description of odd reflection of Takiff superalgebras. The formulation  of typical representations of Takiff superalgebras is given in Subsection \ref{sect::24typical}. In the remaining parts of Subsection \ref{sect::pre}, we develop the Kac induction functors along with some other general technical results and  preparatory tools that are to be used in the sequel. 

In Section \ref{sect::eqv}, we focus on   typical Jordan blocks in the category $\wto_\ell$.     The proof of Theorem \ref{thm::1st} is established in Subsection \ref{sect::322}.  We give in Subsection \ref{sect::ch33} a complete description of characters of irreducible finite-dimensional modules over Takiff superalgebras associated with type I Lie superalgebras. The proof of Theorem \ref{thm::2nd} is established in Subsection \ref{sect::34}. In Remark \ref{rem26::gl11}, one can also
find a detailed example of a complete description of all simple modules for Takiff $\gl(1|1)$.  In Section \ref{sect::4}, we develop Whittaker modules over the Takiff superalgebras and review some tools from the representation theory of finite $W$-algebras associated with the Takiff algebras. We put  these results together to obtain the proof of Theorem \ref{thm::3rd}. Finally, Appendix \ref{App::blocks} is devoted to a description of linkages in the category $\wto_\ell$ of the Takiff $\gl(m|n)$ and Takiff $\mf{osp}(2|2n)$.

\vskip 0.3cm
{\bf  Acknowledgment}.  Chen is partially supported by National Science and Technology Council grants of the R.O.C. and further acknowledges support from  the National Center for Theoretical Sciences,  and he  thanks Volodymyr Mazorchuk and Uhi Rinn Suh for useful and stimulating discussions. Wang is partially supported by the National Natural Science Foundation of China (Nos.12071026, 12471025), and supported by Anhui Provincial Natural Science Foundation 2308085MA01. We are very grateful to Shun-Jen Cheng and   Maria Gorelik for  helpful discussions and  comments.
\vskip 0.3cm
\section{Preliminaries} \label{sect::pre}

Throughout the paper,  the symbols $\Z$ and  $\Z_{\geq 0}$ stand for the sets of all  and non-positive integers, respectively. All vector spaces and algebras  are assumed to be over the field $\C$ of complex numbers.

\subsection{Triangular decompositions}  \label{sect::200} Throughout, we assume that  $\wtg$ is one of
 the following examples of quasi-reductive Lie superalgebras in Kac’s list \cite{Ka1}:
\begin{equation}
\begin{aligned} \label{eq::Kaclist}
&\text{(Basic classical)}~\gl(m|n),~{\mf{sl}}(m|n),~{\mf{osp}}(m|2n), ~D(2,1;\alpha),~ G(3), ~F(4); \\
&\text{(Periplectic)}~\mf{p}(n).
\end{aligned} 
\end{equation} In this article, we exclude the family of queer Lie superalgebras. We also refer to, e.g., \cite{CW12, Mu12} for more details about these Lie superalgebras. Recall that we define $\g:=\wtg_\oa$. We fix a Cartan subalgebra $\h$ of $\g$ and let  $\Phi$ denote the set of all roots. For each root $\alpha\in \Phi$, let $\wtg^\alpha$ be the root space of $\wtg$ corresponding to $\alpha$, that is,
\begin{align}
	&\wtg^\alpha := \{x\in \wtg|~[h, x] =\alpha(h)x, \text{ for all }h\in \h\}. \label{eq::rtspace}
\end{align}   

We fix a {\em triangular decomposition} in the sense of \cite[Section 2.4]{Ma14}:
	\begin{align} \label{eq::tria}
	&\wtg=\widetilde{\mf n}^-\oplus \mf h \oplus \widetilde{\mf n},
	\end{align}   that is,   there exists $H\in \mf h$  such that $$ \wtg^0:=\bigoplus_{\text{Re} \alpha(H)=0}  \wtg^\alpha =\mf h,\quad\widetilde{\mf n}:=\bigoplus_{\text{Re} \alpha(H)>0} \mf g_\alpha,\quad\widetilde{\mf n}^-:=\bigoplus_{\text{Re} \alpha(H)<0} \mf g_\alpha,$$
where $\text{Re}(z)$ denotes the real part of $z\in \mathbb C$. In this case, we define the {\em Borel subalgebra} $\widetilde{\mf b}$ to be $\mf h\oplus \widetilde{\mf n}$. Let  $\Phi_\oa$ and $\Phi_\ob$ stand for the sets of even and odd roots,  respectively. We use $\Phi^+$ and $\Phi^-$ to denote the sets of all  positive roots and negative roots, respectively. Finally, we define  $ \Phi_i^+ :=\Phi^+\cap \Phi_i$  and    $\Phi_i^-:=\Phi^-\cap \Phi_i $, for any $i\in \Z_2$.  

\subsection{Lie superalgebras of type I}  \label{sect::2.1}
In this subsection, we explain our precise setup of the distinguished triangular decompositions of the type I Lie superalgebras $$\wtg= \gl(m|n),\mf{osp}(2|2n) \text{ and }  \pn, \text{  for $m,n\in \Z_{\geq 0}$,}$$  and their type-I gradings 
$$\wtg=\wtg^1\oplus \wtg^0\oplus \wtg^{-1}.$$ We refer to \cite{CW12, Mu12} for further details. The  corresponding Borel subalgebra $\widetilde{\mf b} =\widetilde{\mf b}_\oa\oplus \wtg^1$, as introduced below, will be called  distinguished.  We also define the {\em reverse} Borel subalgebra $\widetilde{\mf b}^r:= \widetilde{\mf b}_\oa \oplus \wtg^{-1}$.

\subsubsection{The general linear Lie superalgebra $\gl(m|n)$}   

The general linear Lie superalgebra $\wtg:=\gl(m|n)$ may be realized as
the space of $(m+n)$ by $(m+n)$ \mbox{complex matrices}  
\begin{align*}
	  \left( \begin{array}{cc} A & B\\
		C & D\\
	\end{array} \right),
\end{align*} where $A\in \C^{m\times m}, B\in \C^{m\times n}, C\in \C^{n\times m}, D\in \C^{n\times n}$, with Lie bracket
given by the super commutator; see also \cite[Section 1.1.2]{CW12} for more details.

The subalgebras $\wtg^1$ and $\wtg^{-1}$ are given by 
\begin{align*}
	&\wtg^0=\g\cong \gl(m)\oplus \gl(n), \\
	&\wtg^1:=
	\left\{\begin{pmatrix}
		0 & B \\
		0 & 0
	\end{pmatrix}\|~B\in \C^{m\times n}\right\}\quad\mbox{and}\quad \wtg^{-1}:=
	\left\{\begin{pmatrix}
		0 & 0 \\
		C & 0
	\end{pmatrix}\|~C\in \C^{n\times m}\right\}.
\end{align*}

 Denote by  $E_{ij}$   the elementary matrix in $\gl(m|n)$ with $(i,j)$-entry $1$ and other entries $0$, for $1\leq i,j\leq m+n$. We fix the Cartan subalgebra $\h:=\bigoplus_{i=1}^{m+n}\C E_{ii}$ consisting of all diagonal matrices. Let 
$\vare_i$ and $\delta_j$ denote the standard dual basis elements for $\h^\ast$  determined by $\vare_j(E_{ii}) =\delta_{ij}$,  $\delta_p(E_{m+q,m+q}) = \delta_{pq}$, and $\vare_j(E_{m+q,m+q}) = \delta_p(E_{ii}) =0$, for any $1\leq i,j\leq m$ and $1\leq p,q\leq n.$
The corresponding sets of even and odd roots are given by
\begin{align*}
	&\Phi_\oa:= \{\vare_i -\vare_j,~\delta_{p}-\delta_q|~1\leq i\neq  j\leq m,~1\leq p\neq  q\leq n\},\\
	&\Phi_\ob:=\{\pm(\vare_i-\delta_p)|~1\leq i\leq m,~1\leq p\leq n\},
\end{align*} respectively.
We fix the positive system $\Phi^+$ as follows:  
\begin{align*}
	\Phi^+&:=\{\vare_i -\vare_j,~\delta_{p}-\delta_q|~1\leq i< j\leq m,~1\leq p< q\leq n\} \\	&\cup\{\vare_i-\delta_p|~1\leq i\leq m,~1\leq p\leq n\}.
\end{align*}
The set $\Phi^-$ of all negative roots is given by  $\Phi^-:= -\Phi^+$.

We let  $\langle\_,\_\rangle:\h^\ast\times \h^\ast\rightarrow \C$ be the bilinear form on $\h^\ast$ induced by the   standard super-trace form, that is, $\langle\vare_i,\vare_j\rangle =\delta_{ij}$ and $\langle\delta_p,\delta_q\rangle =-\delta_{pq}$ and $\langle\vare_i,\delta_p\rangle=0$, for any $1\leq i,j\leq m$ and $1\leq p,q\leq n$.

\subsubsection{The orthosymplectic Lie superalgebra $\mf{osp}(2|2n)$}  \label{sect::212}

 The orthosymplectic Lie superalgebra $\mf{osp}(2|2n)$ is a   
subalgebra of $\mf{gl}(2|2n)$ with the following   realization:
\begin{equation}
	\widetilde{\g}:=\mf{osp}(2\vert 2n)=
	\left\{ \left( \begin{array}{cccc} c &0 & D &E\\
		0 & -c& F & G\\
		-G^t& -E^t & A &B\\
		F^t &D^t & C& -A^t \\
	\end{array} \right)\|~
	\begin{array}{c}
		c\in \C;\,\, D,E,F,G\in \C^{1\times n};\\
		A,B,C\in \C^{n\times n};\\
		B^t =B, C^t=C.
	\end{array}
	\right\}. \label{eq::osp}
\end{equation}
We refer to \cite[Section 1.1.3]{CW12} for further details.  
The type-I grading is given by $\wtg^0:=\g\cong \C \oplus \mf{sp}(n)$ and  
\begin{align*}
&\wtg^1:=
\left\{ \left( \begin{array}{cccc} 0 &0 & D &E\\
	0 & 0& 0 & 0\\
	0& -E^t & 0 &0\\
	0 &D^t & 0& 0 \\
\end{array} \right)\|~
\begin{array}{c}
	 D,E\in \C^{1\times n}
\end{array}
\right\},\\
&\wtg^{-1}:=
\left\{ \left( \begin{array}{cccc} 0 &0 & 0 &0\\
	0 & 0& F & G\\
	-G^t& 0 & 0 &0\\
	F^t &0 & 0& 0 \\
\end{array} \right)\|~
\begin{array}{c}
	F,G\in \C^{1\times n}
\end{array}
\right\}.
\end{align*}

Set $H_\vare:=E_{11}-E_{22},~H_i:=E_{2+i,2+i}-E_{2+i+n,2+i+n},$  for $1\leq i\leq n.$ 
We fix the standard Cartan subalgebra $\mf h:=  \C H_\vare \oplus \bigoplus_{i=1}^n \C H_i$ consisting of diagonal matrices of $\mf{osp}(2|2n)$. Let 
$\{\vare,\delta_1,\delta_2,\ldots,\delta_n\}$ be 
the basis of $\mf h^\ast$ dual to  $\{H_\vare,H_1,H_2,\ldots, H_n\}$.  
	The corresponding even and odd roots  are given by \begin{align*}
		&\Phi_\oa:=\{\pm(\delta_i \pm\delta_j), \pm 2\delta_p|~1\leq i< j <n, ~1\leq p \leq n \},\\
		&\Phi_\ob:=\{\pm(\vare \pm \delta_p)|~1\leq p\leq n\}.
\end{align*} 

We fix the positive system as follows:
\begin{align*}
&\Phi^+:=\{\delta_i \pm\delta_j, 2\delta_p|~1\leq i< j <n, ~1\leq p \leq n \} \cup \{\vare \pm \delta_p|~1\leq p\leq n\}.
\end{align*}
Define  $\Phi^-:= -\Phi^+$, as usual. The  bilinear form $\langle\_,\_\rangle:\mf h^\ast \times \mf h^\ast \rightarrow \C$, which is induced by the super-trace form, is given by $\langle\vare,\vare\rangle =1,~\langle\delta_i,\delta_j\rangle =-\delta_{i,j}, \langle\vare,\delta_i\rangle=0, \text{ for $1\leq i,j\leq n$}.$

\subsubsection{The periplectic Lie superalgebra $\pn$}  \label{sect::223pn}
 The periplectic Lie superalgebra $\pn$ may be realized as a subalgebra inside $\gl(n|n)$ 
 \begin{align}
 	\wtg:=\pn=
 	\left\{ \left( \begin{array}{cc} A & B\\
 		C & -A^t\\
 	\end{array} \right)\| ~ A,B,C\in \C^{n\times n},~\text{$B^t=B$ and $C^t=-C$} \right\}.  \label{eq::max::pn}
 \end{align} See also \cite[Section 1.1.5]{CW12} for more details.

Define the type-I grading of $\pn$ via
\begin{align*}
	\wtg^0=\g\cong \gl(n),\hskip0.3cm
	\wtg^1:=
	\left\{\begin{pmatrix}
		0 & B \\
		0 & 0
	\end{pmatrix}\|B^t=B\right\}\quad\mbox{and}\quad \wtg^{-1}:=
	\left\{\begin{pmatrix}
		0 & 0 \\
		C & 0
	\end{pmatrix}\|C^t=-C\right\}.
\end{align*}
   We  may observe that  $\dim \wtg^{1}>\dim \wtg^{-1}$.

 For any $1\leq i,j\leq n$, define  $e_{ij}:=E_{ij}-E_{n+j,n+i}\in\pn$. We fix the Cartan subalgebra  $\h= \bigoplus_{i=1}^n \C e_{ii}$ consisting of diagonal matrices. Let $\{\vare_1, \vare_2, \ldots, \vare_n\}$ denote the basis for $\h^\ast$ dual to the standard basis $\{e_{11},e_{22},\ldots, e_{nn}\}\subset \h$. The sets of even and odd roots are given by 
 \begin{align*}
 	&\Phi_\oa = \{\vare_i-\vare_j|~1\leq i\neq  j <n\}, \\
 	&\Phi_\ob = \{-\vare_i-\vare_j|~1\leq i< j <n\}\cup \{\vare_i+\vare_j|~1\leq i\leq  j <n\}. 
 \end{align*}

We define the set $\Phi^+$ of positive roots and the set $\Phi^-$ of negative roots as follows:
\begin{align*}
&\Phi^+:=\{\vare_i-\vare_j|~1\leq i< j <n\}  	\cup\{\vare_i+\vare_j|~1\leq i\leq j\leq n\}, \\ 
&\Phi^-:=\{-\vare_i+\vare_j|~1\leq i< j <n\}  	\cup\{-\vare_i-\vare_j|~1\leq i< j\leq n\}.
\end{align*}

 Finally, we define the  bilinear form $\langle\_,\_\rangle:\mf h^\ast \times \mf h^\ast \rightarrow \C$ by declaring that  $\langle\vare_i,\vare_j\rangle = \delta_{ij}$, for $1\leq i,j\leq n.$

\subsection{The category $\wto_\ell$ of the Takiff superalgebra $\wtg_\ell$} \label{sect::22} In this section, unless stated otherwise,  we let  $\wtg$ be a is a basic classical or a periplectic  Lie superalgebra from the list \eqref{eq::Kaclist} with  a   triangular decomposition $\wtg =  \widetilde{\mf n} \oplus \mf h \oplus \widetilde{\mf n}^-$. Recall that the $\ell$-th Takiff superalgebra associated with $\wtg$ is given by $$\wtg_\ell:= \wtg\otimes \C[\theta]/(\theta^{\ell+1}), \text{ where $\theta$ is an even variable}.$$ Recall that we identify $\wtg$ with the subalgebra $\wtg \otimes 1\subset \wtg_\ell$. For any subspace $\mf s \subseteq \wtg$, we put $$\mf s_\ell:= \mf s\otimes \C[\theta]/(\theta^{\ell+1}) =\bigoplus_{i=0}^\ell\mf s_\ell^{(i)},~~~\text{ where }~\mf s_\ell^{(i)} : = \mf s \otimes \mathbb C\theta^i,\text{ for all $i$.}$$ We also set $\mf s^{(i)} =0,$ for $i\geq \ell+1$.
This leads to the triangular decomposition  $\wtg_\ell = \widetilde{\mf n}_\ell \oplus {\mf h}_\ell \oplus  \widetilde{\mf n}_\ell^-.$
For $x\in \wtg$, we set  $x^{(i)}:= x\otimes \theta^i\in \wtg_\ell$, for all  $i\geq 0$. 

 For $0\leq i\leq \ell$, recall that we identify $\h^{(i)}$ with $\h$ via the isomorphism $h\otimes \theta^i\mapsto h$, for $h\in \h$. For $\la \in \h^\ast_\ell$, we let  $\la^{(i)}\in \h^\ast$ denote the weight induced by   $ \la|_{\h^{(i)}}$. 

\subsubsection{Jordan block decomposition of $\wto_\ell$} \label{sect::231} Recall that  $\g$ denotes the even subalgebra of $\wtg$. 
We have the exact restriction and induction  functors 
\begin{align*}
&\Res(\_): \wtg_\ell \mod \rightarrow \g_\ell\mod,~~ \Ind(\_): \g_\ell\mod\rightarrow \wtg_\ell\mod.
\end{align*}
 It follows by \cite[Theorem 2.2]{BF93} (see also \cite{Go00}) that $\Ind(\_)$  is isomorphic to the co-induction functor, up to
 the equivalence given by tensoring with a one-dimensional $\g_\ell$-module. 

Recall that the category $\wto_\ell$ is the full subcategory of $\wtg_\ell\mod$ of modules $M$ with $\Res M$ in the category $\mc O_\ell$ of $\g_\ell$ introduced in \cite{CT23} (see also \cite{MS19,Ch23}).   For each $\kappa\in (\hellgeq)^\ast$,  recall that the Jordan block $\widetilde{\mc O}^\kappa_\ell$ is defined as the full subcategory consisting of objects in $ \wto_\ell$ on which $x-\kappa(x)$ acts locally nilpotently, for each $x\in \hellgeq$. 
\begin{lem} \label{lem::44}
	We have a decomposition 
	\begin{align}
	&\wto_\ell =\bigoplus_{\kappa\in (\hellgeq)^\ast} \wto^\kappa_\ell. \label{eq::16}
	\end{align} Furthermore,   the functors
	$\Res(\_):  \wto_\ell^\kappa \rightarrow \mc O_\ell^\kappa$ and $\Ind(\_): \mc O_\ell^\kappa\rightarrow \wto_\ell^\kappa$ are well-defined,  for any $\kappa\in (\hellgeq)^\ast$.
\end{lem}
\begin{proof}
	The first assertion about the decomposition \eqref{eq::16} follows by a standard argument similar to  the proof of \cite[Corollary 3.8]{Ch23}, we thus omit the details. By definition, $\wto_\ell^\kappa$ can be defined as the full subcategory of $\wto_\ell$ of modules $M$ with $\Res M \in \mc O_\ell^\kappa$. We note that $U((\wtg_\ell)_\ob)$ can be viewed as the exterior algebra $\Lambda((\wtg_\ell)_\ob)$ of $(\wtg_\ell)_\ob$. Since $\Res \Ind M\cong \Lambda((\wtg_\ell)_\ob)\otimes M$ and $\mc O_\ell^\kappa$ is closed under tensoring with finite-dimensional modules, the second assertion follows.
\end{proof}

The idea of such a decomposition in  \eqref{eq::16} originates from the work of Mazorchuk and S\"oderberg \cite[Theorem 11]{MS19} and Chaffe \cite[Corollary 3.8]{Ch23}  for the $1$-st Takiff algebras associated with semisimple Lie algebras, and it was subsequently extended to the case of  general $\ell$-th Takiff algebras by  Chaffe and Topley \mbox{\cite[Lemma  3.4]{CT23}}.

\begin{lem} \label{lem11} 
   Let $E\in \mc O_\ell$ be  finite-dimensional. Suppose that $V\in \mc O_\ell$ has finite length. Then both  $E\otimes V$ and $\Ind(V)$ are   of finite length. 
\end{lem}
\begin{proof} The arguments in the proof of \cite[Theorem 3.5]{Ko75}  go through for the general $\ell$-th Takiff algebras associated with   reductive Lie algebras.   For the convenience of the reader, we  shall sketch the main ideas of the proof. First, for any object $M\in \mc O_\ell$ we let $M^\ast : =\Hom_\C(M,\C)$ denote the full dual representation of $M$ and we define the submodule  $M^\vee:=\bigoplus_{\nu\in {\h^\ast}} (M^{\nu})^\ast \subseteq M^\ast$ as the  restricted dual of $M$; see also Subsection \ref{sect::A1}.  Then $M$ is noetherian (respectively, artinian) if and only if $M^\vee$ is   artinian (respectively, noetherian); see also \cite[Proposition 3.4]{Ko75}.  By using arguments  analogous to those in the proof of \mbox{\cite[Theorem 3.5]{Ko75}}, we can prove that  $(E\otimes V)^\vee$ and $E^\vee \otimes V^\vee$ are isomorphic. Since $V$ has finite length, it follows that  $E^\vee\otimes V^\vee$ is noetherian by a standard argument; e.g., see \cite[Theorem 1.1]{Hu08} or \cite[Proposition 3.3]{Ko75}. This implies that $(E\otimes V)^\vee$ is noetherian. Consequently, $E\otimes V$ is both artinian and noetherian.  This proves the first assertion. The second assertion follows, since $\Res \Ind(V)\cong \Lambda((\wtg_\ell)_\ob) \otimes V$, 
\end{proof}

\subsubsection{Verma and simple modules} \label{sect::222} In this subsection, we introduce some terminologies, following \cite[Section 3]{CT23} and \cite[Section 4]{Sa14}.  
Let $M\in \wto_\ell$ and $\nu\in \h^\ast$. Define the ($\h$-){\em weight space of weight $\nu$} by 
\begin{align*}
&M^\nu:=\{m\in M|~h\cdot m=\nu(h)m,~\text{for all}~h\in \h\}.	
\end{align*} The elements in $M^\nu$ are called {\em weight vectors of weight $\nu$}.   We define the partial order $\leq  $ on $\h^\ast$ as the transitive closure of the relations 
$$\begin{cases}
	\lambda-\alpha \leq \lambda, &\mbox{for $\alpha\in \Phi^+$},\\
	\lambda+\alpha \leq \lambda, &\mbox{for $\alpha\in \Phi^-$}.
\end{cases}$$ 
Let $\text{supp}(M)=\{\nu\in \h^\ast|~M^\nu\neq 0\}$.    It can be proved by mimicking standard arguments from Lie theory that there is a finite subset $X\subset \h^\ast$ such that  (e.g., see \mbox{\cite[Section 1.1]{Hu08}}) $$\dim M^\nu<\infty~\text{and}~\text{supp}(M)\subseteq \bigcup_{\nu \in X}\{\mu\in \h^\ast|~\mu\leq \nu\}.$$

 Let $\la \in \h^\ast_\ell$. 
 A vector $v\in M$ is called a ($\widetilde{\mf b}_\ell$-){\em highest weight vector of weight $\la$} provided that $\widetilde{\mf n}_\ell\cdot v=0$ and $h\cdot v =\la(h)v,~\text{for all }h\in \h_\ell.$

 A $\wtg_\ell$-module is called a ($\widetilde{\mf b}_\ell$-){\em highest weight module of highest weight $\la$} if it is generated by a highest weight vector of weight $\la$. 
We define the Verma module $$\wtV({\la}):=U(\wtg_\ell)\otimes_{U(\widetilde{\mf b}_\ell)} \C_{\la},$$
where $\C_{\la}$ is a one-dimensional $\widetilde{\mf b}_\ell$-module such that $\widetilde{\mf n}_\ell\cdot \C_\la =0 $ and $ (h-\la(h))\cdot\C_\la=0$, for all $h\in \h_\ell$. We may observe that $\widetilde{M}(\la)$ is a highest weight module of highest weight $\la$.  By definition, each highest weight module is a  quotient of  a Verma module. Recall that $\la^{\scriptscriptstyle (\geq 1)}$ denotes the restriction   $\la|_{\hellgeq}$, for any weight $\la\in \h^\ast_\ell$. We collect some basic properties in the following lemma.
 
\begin{lem} \label{lem::1}
 We have 
  \begin{itemize}
	\item[(i)] Let $\la \in \h^\ast_\ell$. Then the Verma module  $\wtV({\la})$ has a  simple top, which we denote by $\wtL(\la)$. Both modules lie in the Jordan block $\wto_\ell^{\la^{\scriptscriptstyle (\geq 1)}}$.  Furthermore,	the center $Z(\wtg_\ell)$  acts on $\widetilde{M}(\la)$ and $\wtL(\la)$ by some character associated with $\la$, which we denote by $\widetilde{\chi}_\la(\_): Z(\wtg_\ell)\longrightarrow \C$.
 	\item[(ii)] $\{\wtL(\la)|~\la\in \h^\ast_\ell\}$ is a complete list of non-isomorphic simple modules in $\widetilde{\mc O}_\ell$.
 	\item[(iii)] Every module in $\wto_\ell$ admits a finite filtration subquotients of which are highest weight modules. 
\end{itemize}
\end{lem}
\begin{proof}  The assertions in Part (i) follow by a standard argument, we omit the proof. It remains to prove the assertion in Part (iii).  Let $M\in \widetilde{\mc O}_\ell$. Then there exists a finite-dimensional $\widetilde{\mf b}_\ell$-submodule $V\subseteq M$  such that $M = U(\wtg_\ell)V$. We will proceed with the proof by induction on the dimension of $V$. Then it follows from standard highest weight theory (cf. \cite[Subsection 1.5.3]{CW12}) that $V$ admits a one-dimensional simple $\widetilde{\mf b}_\ell$-submodule $V_\la\subseteq V$, which is spanned by a $\h_\ell$-weight vector of weight $\la \in \h^\ast_\ell$. This induces a non-zero epimorphism from $\widetilde{M}(\la)$ to $U(\wtg_\ell)V_\la$. By induction, $M/U(\wtg_\ell)V_\la$ has a finite filtration subquotients of which are highest weight modules. This proves \mbox{Part (iii)}. We note that Part (ii) is a direct consequence of Part (iii).   This completes the proof.
\end{proof}

\begin{rem}  It is worth pointing out that, in general, the Verma modules for Takiff algebras do not have a finite composition series; see \cite{MS19} for Takiff $\mf{sl}(2)$. Therefore, the 
	categories $\mc O_\ell$ and $\wto_\ell$ fail to be artinian, in general.    
\end{rem}
 
\subsubsection{Characters and composition multiplicities} \label{sect::223}
For $M\in\wto_\ell$, we define the {\em character} of $M$ as 
\begin{align*}
&\ch M:=\sum_{\nu\in \h^\ast} \dim(M^\nu)e^{\nu}.
\end{align*}  
 
The following result is an analogue of  \cite[Lemma 6.1]{CT23} for Takiff superalgebras, which    in turn goes back to the work of Mazorchuk and  S\"oderberg \mbox{\cite[Proposition 8]{MS19}}. 
\begin{lem} \label{lem::chmul} Fix $\kappa\in  (\hellgeq)^\ast$  and  $M\in \wto_\ell^\kappa$.  Then there are unique non-negative integers $\{k_\la(M)|~\la\in \h^\ast_\ell \text{ with }\la^{\scriptscriptstyle (\geq 1)} =\kappa\}$ such that $\ch M =\sum_{\la\in \h^\ast_\ell} k_\la(M)\ch \wtL(\la).$
\end{lem}
\begin{proof}
The proof follows mutatis mutandis from the proof of  \cite[Lemma 6.1]{CT23} (see also \cite[Proposition 8]{MS19}). 
\end{proof}
Let $M\in \wto_\ell$.  The number $k_\la(M)$ will be called the {\em composition multiplicity} of $\wtL(\la)$ in  $M$ by analogy with that of modules over Takiff algebras; see also \cite{MS19,CT23}. We use the conventional notation to denote the numbers, namely, we define 
 $$[M:\wtL(\la)]:=k_\la(M),~ \text{for any}~ M\in \wto_\ell\text{ and }\la\in \h^\ast_\ell.$$ We  give a representation-theoretical interpretation of composition multiplicity as follows, which is a  super analogue of \cite[Lemma 6.2]{CT23}. 
	\begin{prop} \label{prop::12}
		Let $M\in \wto_\ell$. Then there exists a descending filtration 
		\begin{align}
			M=M_0\supseteq  M_1\supseteq M_2\supseteq \cdots,  \label{lem7eq::16}
		\end{align} such that 
	\begin{itemize}
		\item[(a)] $M_i/M_{i+1}$ are simple highest weight modules, for all $i\geq 0$;  
		\item[(b)] $\bigcap_{j\geq 0} M_j =0.$
	\end{itemize}  
Furthermore,   $\wtL(\la)$ appears as a quotient $M_i/M_{i+1}$ precisely $[M:\wtL(\la)]$ times,  for  each $\la\in \h_\ell^\ast$.
	\end{prop}
	\begin{proof}
    First, suppose that $M= \Ind(V),$ for some $V\in \mc O_\ell$. By  \cite[Lemma 6.2]{CT23}, there exists  a descending filtration \begin{align} &V=V_0\supseteq  V_1\supseteq V_2\supseteq \cdots,  \label{eq::188}      \end{align} such that $V_i/V_{i+1}$ are simple, for all $i\geq 0$, and $\bigcap_{j\geq 0} V_j =0.$  	 Applying   the induction functor $\Ind(\_)$ to \eqref{eq::188}, then the conclusion follows by  Lemma \ref{lem11}. Next, assume that $M\in \wto_\ell$ is arbitrary and put  $V:=\Res M \in \mc O_\ell$.  We have already shown that there exists a filtration $\Ind (V) = K_0 \supseteq K_1 \supseteq K_2 \supseteq \cdots$  satisfying conditions (a) and (b). There is a canonical epimorphism $\pi: \Ind(V) \rightarrow M$  by adjunction, and as a result, we obtain a filtration  	\begin{align*} 	 &M=\pi(K_0)\supseteq \pi(K_1)\supseteq   \pi(K_2)\supseteq \cdots,   \end{align*} can be refined to be a filtration $M=M_0\supseteq  M_1\supseteq M_2\supseteq \cdots$ satisfying conditions (a) and (b). Furthermore, by Lemmas \ref{lem::1}, \ref{lem::chmul} and an argument similar to one used in \cite[Lemma 6.2]{CT23}, it follows that each $\wtL(\la)$ appears as a quotient $M_i/M_{i+1}$ precisely $[M: \wtL(\la)]$ times.  This completes the proof. 
\end{proof}

  By  {\em component} of $M$, we mean a simple module $\wtL(\la)$ with positive composition multiplicity in $M$, that is,  $[M:\wtL(\la)]>0$.  We let  $\approx$ be the minimal equivalence relation on $\h^\ast_\ell$ generated by relation $\la \approx\mu$ if we have two components $\wtL(\la), \wtL(\mu)$ in an indecomposable module of $\wto_\ell$.  In addition, let $\Lambda\in \h^\ast_\ell$ be an equivalence class, then  an indecomposable module $M\in \wto_\ell$ is said to be of {\em type $\Lambda$} if all components $\wtL(\la)$ of $M$ satisfy $\la \in \Lambda$. These terminologies are given here  inspired by composition multiplicity of modules over affine Lie algebras (e.g., see  \cite[Definition 2.1.11]{K02}).

On the other hand, we define a relation $\rightarrow$ on $\h^\ast_ \ell$ by declaring that 
\begin{align}
&\la\rightarrow \mu\text{ for }\la,\mu\in \h^\ast_\ell, \text{ if }[\widetilde{M}(\la):\wtL(\mu)] \neq 0. \label{eq::17}
\end{align} We note that $ \la\rightarrow\mu$ implies that $\la^{\scriptscriptstyle (\geq 1)}=\mu^{\scriptscriptstyle (\geq 1)}$. Let $\sim$ be the equivalence relation on $\h^\ast_\ell$ generated by $\rightarrow$.  The following proposition is a standard result in representation theory.

\begin{prop} \label{lem::blockdec} Let $\wtg$ be any basic classical Lie superalgebra. Then, for any weights $\la, \mu\in \h^\ast_\ell$, we have $$\la \approx \mu\Leftrightarrow \la \sim \mu.$$ In particular, if $M\in \wto_\ell$ then there are equivalence classes $\Lambda_1,\ldots,\Lambda_k \in \h^\ast_\ell/\sim$ such that 	 $M=\bigoplus_{i=1}^k M_{\Lambda_i}$, for some indecomposable submodules $M_{\Lambda_1},\ldots, M_{\Lambda_k}$ of types $\Lambda_1,\ldots,\Lambda_k$, respectively.    
\end{prop}
\begin{proof}
	See   Appendix   \ref{App::blocks}. 
\end{proof}

\subsection{Odd reflections} \label{sect24::OddRef}
As is well known,  the Borel subalgebras of a basic classical or a strange Lie superalgebra are not conjugate in general under the action of the Weyl group. In \cite{LSS86}, Leites, Saveliev, and Serganova develop the method of {\em odd reflecitons} to relate non-conjugate Borel subalgebras and their positive systems; see also \cite[Lemma 1 and Section 2.2]{PS89} and \cite[Subsection 1.4]{CW12}.  The goal of this subsection is to develop  analogues of the odd reflections for the Takiff superalgebra associated to a basic classical or a strange Lie superalgebra. Then we describe the highest weights of simple highest weight modules with respect to various Borel subalgebras  connected by the odd reflections.

\subsubsection{Basic classical case} \label{sect::oddreflbasic} In this subsection, we  assume that $\wtg$ is a basic classical Lie superalgebra. 
Let $ \widetilde{\mf b}$ be an arbitrary Borel subalgebra of $\wtg$ with a positive system $\Phi^+ \subseteq \mf h^\ast$. Suppose that $\alpha\in \Phi^+$ is an isotropic odd simple root.  Denote by $\widetilde{\mf b}(\alpha)$ the new Borel subalgebra spanned by root vectors associated with roots in the new positive system  $\Phi^+(\alpha):=\Phi^+\backslash\{\alpha\}\cup \{-\alpha\}$.  
Following \cite{PS89}, we say that the Borel subalgebras $\widetilde{\mf b}$ and $\widetilde{\mf b}(\alpha)$ are  {\em connected by an odd reflection associated with $\alpha$}. Then, any two Borel subalgebras with the same even parts are connected by a chain of odd reflections. 

We choose non-zero root vectors  $e_\alpha\in\wtg^{\alpha}$, $f_\alpha\in \wtg^{-\alpha}$ and define $h_\alpha :=[e_\alpha, f_\alpha]$ such that $\nu(h_\alpha) = \langle \nu, \alpha \rangle$ for $\nu \in \h^\ast$. Recall that we set  $x^{(i)}:= x\otimes \theta^i\in \wtg_\ell$, for  $x\in \wtg$ and $i\geq 0$.    

\begin{prop} \label{thm::Oddrefl}  
Suppose that $\wtg$ is basic classical. Let $L$ be a simple $\wtg_\ell$-module in $\wto_\ell$ with $\widetilde{\mf b}_\ell$-highest weight vector $v_\la\in L$ of $\widetilde{\mf b}_\ell$-highest weight $\la =(\la^{(0)},\la^{(1)},\ldots,\la^{(\ell)})\in \h^\ast_\ell$. Let $\alpha$ be an isotropic odd simple root. Then  
\begin{itemize}
    \item[(i)] Suppose that   $\langle \la^{(k)},\alpha\rangle\neq 0$ and $$\langle \la^{(k+1)},\alpha\rangle= \cdots =\langle \la^{(\ell)},\alpha\rangle =0,$$ for some $0\leq k\leq \ell$. Then   $f_\alpha^{(0)}f_\alpha^{(1)}\cdots f_\alpha^{(k)}v_\la$ is a $\widetilde{\mf b}(\alpha)_\ell$-highest weight vector of $L$. 
    \item[(ii)]  Suppose that   $\langle \la^{(k)},\alpha\rangle= 0$, for all $0\leq k\leq \ell$. Then $v_\la$ is a $\widetilde{\mf b}(\alpha)_\ell$-highest weight vector of $L$. 
\end{itemize}
\end{prop}
\begin{proof} 
First, suppose that the condition in Part (i) holds. Let $v_\la' :=f_\alpha^{(0)}f_\alpha^{(1)}\cdots f_\alpha^{(k)} v_\la$.  To prove that  $ v'_\la\neq 0,$ we calculate \begin{align*}
&e_\alpha^{(k)}v_\la' =  \sum_{i=0}^k c_i f_{\alpha}^{(0)}\cdots f_{\alpha}^{(i-1)} h_\alpha^{(k+i)} f_{\alpha}^{(i+1)}\cdots f_{\alpha}^{(k)} v_\la ,
\end{align*} for some $c_i =\pm 1$. Since $h_\alpha^{(k+i)}v_\la  = \langle \la^{(k+i)}, \alpha\rangle v_\la=0,$ for any $i>0$, it follows that $e_\alpha^{(k)}v_\la' =   h_\alpha^{(k)} f_\alpha^{(1)}\cdots f_\alpha^{(k)} v_\la = \langle \la^{(k)},\alpha \rangle  f_\alpha^{(1)}\cdots f_\alpha^{(k)} v_\la\neq 0$  and thus $v_\la'\neq 0$.

To prove that $v_\la'$ is a $\widetilde{\mf b}(\alpha)_\ell$-highest weight vector of $L$, we first show that $f_\alpha^{(i)}v_\la' =0,$ for all $0\leq i\leq \ell$. This is obvious for $0\leq i\leq k$, since $\alpha$ is isotropic. 
For $i>k$, we shall prove that $ f_\alpha^{(i)}v_\la=0$, which implies that $ f_\alpha^{(i)}v'_\la=0$. Suppose on the contrary that $f_\alpha^{(i)}v_\la$ is non-zero. We may observe that, for any $\beta \in \Phi^+(\alpha)\cap \Phi^+$ and $0\leq j\leq \ell$, it follows that  $e_\beta^{(j)}f_\alpha^{(i)}v_\la=0$ since  $\beta -\alpha$ is either not a root or a root in $\Phi^+(\alpha)\cap \Phi^+$. Also, we may note that $e_\alpha^{(j)}f_\alpha^{(i)}v_\la= h_\alpha^{(i+j)}v_\la = \langle \la^{(i+j)}, \alpha \rangle v_\la =0,$ since $i+j>k$. Consequently,     $f^{(i)}_\alpha v_\la$ is a non-zero $\widetilde{\mf b}_\ell$-highest weight vector of $L$, a contradiction. 

Next, we may note that, for any $\beta \in \Phi^+(\alpha)\cap \Phi^+$ and $0\leq i\leq \ell$, since $\beta-r\alpha$ is either not a root or a root in $\Phi^+(\alpha)\cap \Phi^+$ for any   $r>0$, it follows that $e^{(i)}_\beta v_\la'=0$. Furthermore, by a straightforward calculation we may infer that $v'_\la$ is a $\h_\ell$-weight vector.   This proves the assertion in Part (i). The conclusion in Part (ii) can be proved by a similar argument. This completes the proof.
\end{proof}

\subsubsection{Periplectic case} \label{sect::oddrefl}  The goal of this subsection is to describe the odd reflections for the Takiff superalgebras associated to periplectic Lie superalgebras.  In this subsection, we assume that  $\wtg= \pn$. Suppose that $\mf B$ and $\mf B'$ are two Borel subalgebras of $\pn$  with the same even parts. Following \cite{PS89}, they are said to be {\em connected by an  inclusion} if  $\mf B $ is a co-dimensional one subalgebra of $  \mf B'$. Otherwise, we say that they are {\em connected by an odd reflection associated with an odd root $\gamma$} provided that the following are satisfied:  
\begin{enumerate}
    \item[$\bullet$ ]$\mf B\cap \mf B'$ is a co-dimensional one subalgebra of both $\mf B$ and ${\mf B}'$;
     \item[$\bullet$ ]   $c\gamma$ is a root of ${\mf B}$ and $-c\gamma$  is a root of $\mf B'$, for some $c=\pm 1.$
\end{enumerate} 

Now, we let $\widetilde{\mf b}$ be the distinguished Borel subalgebra from Subsection \ref{sect::223pn}. By \cite[Lemma 1 and Section 2.2]{PS89}, there is a  chain of odd reflections and inclusions $$\mf B^0:=\widetilde{\mf b}^{r}, \mf B^1,
\mf B^2, \cdots, , \mf B^{q-1}, \mf B^q:= \widetilde{\mf b},$$
connecting the Borel subalgebras $\widetilde{\mf b}^r$ and $\widetilde{\mf b}$ by the sequences of roots \[ \{\gamma_0,
\gamma_1, \cdots , \gamma_{q-1} \} := \{ 2\epsilon_1,\text{ }
\epsilon_1 +\epsilon_2 ,\text{ } 2\epsilon_2 ,\text{ } \epsilon_1
+\epsilon_3,\text{ } \epsilon_2 +\epsilon_3, \text{ }
2\epsilon_3,\text{ } \cdots, \text{ } \epsilon_{n-1}
+\epsilon_n,\text{ } 2\epsilon_n \}
\] such that $\gamma^p$ is a root of ${\mf B}^{p+1}$ and it is not a root of ${\mf B}^p$,  for each $p$.  

For any $p$, we may note that ${\mf B}^p$ and ${\mf B}^{p+1}$ is connected by an odd reflection if and only if  $\gamma_p =\vare_i+\vare_j$ for some $i<j$. In this case, we set $\ov{\gamma^p}: =\vare_i-\vare_j$.

\begin{prop} \label{thm::Oddreflforpn} Suppose that $\wtg=\pn$ and keep notations as above. Let $L$ be a simple $\wtg_\ell$-module in  $\wto_\ell$ with $\mf B^p_\ell$-highest weight vector $v_\la\in L$ of $\mf B^p_\ell$-highest weight $\la =(\la^{(0)},\la^{(1)},\ldots,\la^{(\ell)})\in \h^\ast_\ell$, for some $0\leq p< q$. Then we have 
\begin{itemize}
    \item[(i)] Suppose that $\gamma^p =\vare_i+\vare_j$, for some $i<j$. If   $\langle \la^{(k)},\ov{\gamma^p}\rangle\neq 0$ and $$\langle \la^{(k+1)},\ov{\gamma^p}\rangle= \cdots =\langle \la^{(\ell)},\ov{\gamma^p}\rangle =0,$$ for some $0\leq k\leq \ell$. Then   $f_\alpha^{(0)}f_\alpha^{(1)}\cdots f_\alpha^{(k)}v_\la$ is a $\mf B^{p+1}_\ell$-highest weight vector of $L$. Otherwise, if $\langle \la^{(k)},\ov{\gamma^p}\rangle= 0$, for all $0\leq k\leq \ell$. Then $v_\la$ is a ${\mf B}_\ell^{p+1}$-highest weight vector of $L$.
    \item[(ii)]  Suppose that ${\mf B}^p$ and ${\mf B}^{p+1}$ are connected by an inclusion. Then $v_\la$ is a ${\mf B}_\ell^{p+1}$-highest weight vector of $L$.
\end{itemize}
\end{prop}
\begin{proof}
The arguments in the proof of Proposition \ref{thm::Oddrefl} can be adapted to prove the assertion in Part (i), and so we omit it. The conclusion in Part (ii) follows by considering  weights of $L$. This completes the proof. 
\end{proof}

\subsection{Typical  representations} \label{sect::24typical}
In this subsection, we give a construction of the central element $\widetilde{\Omega}$ from Theorem \ref{thm::1st} and define   typical representations of Takiff superalgebras associated with  basic classical Lie superalgebras. Let $\mf s$ be a finite-dimensional complex Lie superalgebra. Denote by $S(\mf s)$ and $S(\mf s)^{\mf s}$  the symmetric algebra  of $\mf s$ and its subalgebra of $\mf s$-invariants under the adjoint action, respectively. Then $S(\mf s)$ and $U(\mf s)$ are isomorphic as $\mf s$-modules under the adjoint action via the the canonical super-symmetrization map $\omega_{\mf s}: S(\mf s)\rightarrow U(\mf s)$; see, .e.g, \cite[Subsection 0.5]{Ser99} or \cite[Proposition 2.19]{CW12}. It is clear that $\omega_{\mf s}$ restricts to an isomorphism  from $S(\mf s)^{\mf s}$ to  $Z(\mf s)$. 

 \subsubsection{}   Suppose that $\wtg$ is a basic classical Lie superalgebra  from \eqref{eq::Kaclist} with a Borel subalgebra $\widetilde{\mf b}$ and a Cartan subalgebra $\mf h$ from a triangular decomposition in \eqref{eq::tria}.  Fix  an even  non-degenerate invariant super-symmetric bilinear $\langle\_,\_\rangle$ on $\wtg$. Then, it induces a non-degenerate bilinear form on $\mf h^\ast$, and furthermore, it allows us to identify $S(\wtg)^{\wtg}\cong S(\wtg^\ast)^{\wtg}$ and $\theta: S(\mf h)\cong S(\mf h^\ast)$. Transferred via these identifications, the restriction map from $S(\wtg^\ast)$ to $S(\mf h^\ast)$ induces a homomorphism of algebras $\eta: S(\wtg)\rightarrow S(\mf h^\ast)$. 
 
 Let $\mf{hc}^\ast = \theta \circ \mf{hc}:  Z(\wtg)\rightarrow U(\h^\ast)$ be the composition of the isomorphism $\theta$ above and the associated Harish-Chandra homomorphism $\mf{hc}: Z(\wtg)\rightarrow U(\mf h)$ (see, e.g., \cite[Subsection 2.2.3]{CW12}).  By the standard theory of the invariants polynomials on complex finite-dimensional simple Lie superalgebras,  it follows that  $\eta(S(\wtg)^{\wtg}) = \mf{hc}^\ast(Z(\wtg))$ and its contains the following element  $$P =\prod_{\alpha\in \ov{\Phi}^+_\ob} \alpha \in S(\mf h^\ast),$$ where $\ov{\Phi}^+_\ob$ denotes the  set of all positive odd isotropic roots;  see   \cite[Main Theorem]{Ser99}, \cite[Section 4]{Go04}, \cite{Ka84} and \cite[Subsection 2.2]{CW12}. Also, this implies that the central element $(\mf{hc}^\ast)^{-1}(P)$ acts on any highest weight $\wtg$-module with highest weight $\nu\in \mf h^\ast$ as the scalar $$\prod_{\alpha\in \ov{\Phi}^+_\ob} \langle \nu+\widetilde\rho, \alpha\rangle,$$ where  $ \widetilde\rho\in \mf h^\ast$  denotes the associated Weyl vector of $\wtg$, that is, $\widetilde\rho= \frac{1}{2}({\sum_{\alpha\in \Phi_\oa^+}\alpha}-\sum_{\beta\in \Phi_\ob^+}\beta).$ 

Next, we form the Takiff superalgebra $\wtg_\ell$ associated with $\wtg$. Then the linear map $\_\otimes \theta^\ell: \wtg\rightarrow \wtg_\ell$, sending $x$ to $  x\otimes \theta^\ell$,  induces an inclusion $\iota$ from $S(\wtg)^{\wtg}$ to $S(\wtg_\ell)^{\wtg_\ell}$. This leads to an inclusion of algebras ${\omega}_{\wtg_\ell}\circ \iota: S(\wtg)^{\wtg} \hookrightarrow Z(\wtg_\ell)$, where $\omega_{\wtg_\ell}$ denotes the canonical symmetrization map of $\wtg_\ell$ as above. Now, we can define the element $$\widetilde{\Omega}: = \omega_{\wtg_\ell} \circ \iota (\eta^{-1}(P)) \in Z(\wtg_\ell).$$ 

    Recall from Lemma \ref{lem::1} that $\widetilde{\chi}_\la$ denotes the central character of $\wtL(\la)$, for $\la \in \mf h^\ast_\ell$. By construction, we have  $$\widetilde{\chi}_\la(\widetilde{\Omega}) =\prod_{\alpha\in \ov{\Phi}^+_\ob} \langle \la^{(\ell)}, \alpha\rangle. $$ Here we recall  $\la^{(\ell)}$  denotes the restriction of $\la$ to $\mf h\otimes \theta^\ell\cong \mf h^\ast$ as before. A central character $\widetilde{\chi}: Z(\wtg_\ell)\rightarrow \C$ is called {\em typical} provided that $\widetilde{\chi}(\widetilde{\Omega})\neq 0$. A  $\wtg_\ell$-module $M$ is referred to as a {\em typical representation} if $M$ admits a typical central character,  namely, there is a typical central character $\widetilde{\chi}$ such that $z$ acts on $M$ as the scalar $ \widetilde{\chi}(z)$, for any $z\in Z(\wtg_\ell)$.

\subsubsection{} \label{sect::252} 
 In the paper, we define the Weyl group $W$ of the Takiff superalgebra $\wtg_\ell$ to be the Weyl group of $\g$, for a basic or a periplectic Lie superalgebra $\wtg$.  
 In this subsection, we review from \cite{G95} the results regarding the Harish-Chandra homomorphism of the Takiff algebras  associated with reductive Lie algebras.  Following \cite[Subsection 2.4]{G95}, we extend the natural action of $W$ on $\h^\ast$ to $\h^\ast_\ell$ diagonally, that is,  $$w\la = (w\la^{(0)},w\la^{(1)},\ldots,w\la^{(\ell)}),\text{ for $w\in W$ and }\la\in \h^\ast_\ell.$$   Let   $\rho\in \h^\ast$ denotes the Weyl vector of $\g$ (i.e., the half sum of positive roots in $\g$) and $\rho_\ell\in \h_\ell^\ast$ is an associated weight determined by $(h\in \h)$: 
 \begin{align}
 	&	\rho_\ell(h^{(i)})=  	  
    \begin{cases} (\ell+1) \rho(h), &\mbox{ for~$i=0$;}\\  	    	0,& \mbox{ otherwise}. \end{cases}
 \end{align}  We define the following  {\em dot-action} of $W$ on $\h^\ast_\ell$ by declaring that 
 \begin{align} 
 &w\cdot \la = w(\la+\rho_\ell) -\rho_\ell, \text{ for $w\in W$ and $\la \in \h^\ast_\ell$.} \label{eq::dotaction}
 \end{align}

Define   $\psi_{\h_\ell}:Z(\g_\ell)\rightarrow  {U}(\mathfrak{h}_\ell)$  to be the projection associated with the decomposition $U(\g_\ell) = U(\mf h_\ell)\oplus\left(\mf n_\ell^{-} U(\g_\ell)+U(\g_\ell)\mf n_\ell\right).$ We regard $U(\h_\ell)$ as the algebra of polynomial functions on $\h^\ast_\ell$ and define an automorphism $\tau$ of $U(\h_\ell)$ by $\tau(f)(\la):=f(\la-\rho_\ell),$ for all $f\in U(\h_\ell)$ and $\la \in \h^\ast_\ell$.   Following \cite{G95}, we define the   Harish-Chandra homomorphism $\mf{hc}_\ell$ for $\g_\ell$ to be the composition 
$$\mf{hc}_\ell:=\tau \psi_{\mf h_\ell}: Z(\g_\ell)\rightarrow  U(\h_\ell).$$ The following lemma   taken from \cite[Theorem~4.7]{G95} will be useful in sequel.
\begin{lem}[Geoffriau] \label{lem11::G} The   Harish-Chandra homomorphism 
$\mf{hc}_\ell$ is an embedding from $Z(\g_\ell)$ to the algebra  $U(\h_\ell)^W$ consisting of $W$-invariants in $U(\h_\ell)$. 
\end{lem}  

\begin{rem}\label{rem13::typcon}
    Suppose that $\wtg$ is basic classical. We can also extend the Weyl vector $\widetilde{\rho}$ of $\wtg$ to an element in $\h_\ell^\ast$ in a similar fashion. Then the typical condition of a weight $\la\in \h^\ast_\ell$ of the  Takiff superalgebra $\wtg_\ell$ with $\ell \geq 0$ can be uniformly given as follows 
\begin{align*}
&\la \text{ is typical }\Leftrightarrow \prod_{\alpha\in \ov{\Phi}^+_\ob}\langle(\la+\widetilde{\rho}_\ell)^{(\ell)},\alpha\rangle\neq 0. 
\end{align*}  
\end{rem}

\subsection{Kac induction functor} \label{sect::24}
  In this subsection, unless stated otherwise, we let $\wtg =\gl(m|n), \mf{osp}(2|2n),$ or $\pn$ with the type-I grading $\wtg = \wtg^{-1}\oplus \wtg^0\oplus \wtg^1$ the triangular decompositions from  Subsection \ref{sect::2.1}.   Set  $\wtg_\ell^{\scriptscriptstyle \geq 0}:=\wtg_\ell^0+\wtg_\ell^1$. For any $V\in \g_\ell\mod$,  we may extend the $V$  to a $\wtg_\ell^{\scriptscriptstyle \geq 0}$-module by letting $\wtg_\ell^1V=0$. Define the Kac induced module  $$K(V) : =U(\wtg_\ell){\otimes}_{U(\wtg_\ell^{\scriptscriptstyle \geq 0})}V.$$ This gives rise to the exact {\em Kac induction functor} $K(\_):\g_\ell\mod\rightarrow\wtg_\ell\mod$.
 
 \subsubsection{$\Z$-gradings of  Kac modules}
 	Define the following grading operators:
 \begin{align}
 	&d^{\widetilde{\g}}:=\left\{ \begin{array}{lll} \sum\limits_{i=m+1}^{m+n}E_{ii},\quad \text{~for $\widetilde{\g}= \gl(m|n)$;} \\ 
 		H_\vare = E_{11}-E_{22},\quad \text{~~~for $\widetilde{\g}= \mf{osp}(2|2n)$;}\\
 		\sum\limits_{i=1}^n e_{ii}= \sum\limits_{i=1}^n(E_{ii}-E_{n+i,n+i}), \quad \text{~~for $\widetilde{\g} =\pn$.} \end{array} \right. \label{eq::graop}
 \end{align}

 We view the subalgebras  $U(\wtg_\ell^{\pm 1 })$ of $U(\wtg_\ell)$ as 
 exterior algebras $\Lambda(\wtg_\ell^{\pm1})$ of $\wtg_\ell^{\pm 1}$. For $k\in \Z_{\geq 0}$, we set  $\Lambda^k(\wtg_\ell^{\pm 1})$ to be the subspace spanned by all $k$-th exterior powers in  $\Lambda(\wtg_\ell^{\pm})$. Suppose that  $V$ is a simple $\g_\ell$-module, then $V$ admits a central character of $\g_\ell$ by Dixmier’s theorem \cite[Proposition 2.6.8]{Di96}. Thus, $K(V)$ admits a central character of $\wtg_\ell$. Suppose that the central character of $V$ is of the form $\chi_\la$ for some $\la \in \h^\ast_\ell$, then $K(V)$ admits a typical central character if and only if $\la$ is typical.
 
 We note that $d^{\wtg}\in Z(\g_\ell)$, which means that  $d^{\wtg}$ acts on $V$ as a scalar $d^{\wtg}_V\in \C$. In this case,  the restriction $\Res K(V)$ decomposes as a finite direct sum of generalized  eigenspaces $\Lambda^k(\wtg_\ell^{-1})\otimes V$ of $d^{\wtg}$   with eigenvalues $d^{\wtg}_V-k_{\wtg}$, where   
 \begin{align*}
 	&	k_{\widetilde{\g}}:=  	    \begin{cases} k, &\mbox{ for~$\widetilde \g =\gl(m|n)$ or $\mf{osp}(2|2n)$;}\\  	    	2k,&\mbox{ for $\widetilde \g = \pn$}. \end{cases}
 \end{align*} Also, we identify $V$ as the $\g_\ell$-submodule $\C1_{\Lambda(\wtg_\ell^{-1})}\otimes V$ of $\Res K(V)$.
 
 For any $M\in \wtg_\ell\mod$, we define the $\g_\ell$-submodule  
 $M^{\wtg_\ell^1} = \{m\in M|~\wtg_\ell^1\cdot m=0\}.$    We note that the   functor $(\_)^{\wtg_\ell^1}:\wtg_\ell\mod \rightarrow\g_\ell\mod$ of taking $\wtg_\ell^1$-invariants of $\wtg_\ell$-modules is the  right adjoint of the Kac functor $K(\_)$.  Furthermore, for any $\kappa\in (\hellgeq)^\ast$  they restricts well-defined functors   \begin{align*}
		&K(\_):  \mc O_\ell^\kappa \rightarrow \wto_\ell^\kappa~\text{and } (\_)^{\wtg_\ell^1}: \wto_\ell^\kappa\rightarrow\mc O^{\kappa}_\ell  
	\end{align*} such that $K(M(\la))\cong \widetilde{M}(\la)$, for any $\la\in \h^\ast_\ell$.  
	 By Lemma \ref{lem11}, if $V\in \mc O$ has finite length, then so are $K(V)$.

\subsubsection{Simplicity of Kac modules} \label{sect::262Kac} For any type I Takiff superalgebra $\wtg_\ell$, we define $\Lambda^{\text{max}}(\wtg_\ell^{\pm 1}) := \Lambda^{\dim \wtg_\ell^{\pm 1}}( \wtg_\ell^{\pm 1})$ and choose  two non-zero  vectors $X^{\pm} \in \Lambda^{\text{max}}(\wtg_\ell^{\pm 1})$.   Assume that  $\wtg$ is either $\gl(m|n)$ or $\mf{osp}(2|2n)$. By using the fact that  $\Phi^+ =-\Phi^-$, and combining it with arguments  analogous to those in the proof of   \mbox{\cite[Lemma 6.5]{CM21}},  we can prove that 
\begin{align}
	&X^+X^-= \Omega' +\sum_{i} x_i^-r_ix_i^+, \label{Propeq::9}
\end{align}
for some $\Omega'\in Z(\g_\ell)$, $r_i\in U(\g_\ell)$, and $x_i^{\pm}\in \wtg_\ell^{\pm 1}\Lambda(\wtg_\ell^{\pm 1})$.  	
The following is a generalization of \cite[Theorem 6.7]{CM21}.  
\begin{thm} \label{prop::Kacsimple}
 Suppose that $\wtg=\gl(m|n)$ or $\mf{osp}(2|2n)$. Let $V$ be a simple $\g_\ell$-module and $\chi: Z(\g_\ell)\rightarrow \C$ be the central character of $V$. Then the following are equivalent:
	\begin{itemize}
		 		\item[(i)] $K(V)$ is simple. 
		\item[(ii)]  $\Omega'V\neq 0$.
        \item [(iii)] $\chi(\Omega')\neq 0$.
       \end{itemize}
\end{thm}
\begin{proof} 
Mutatis mutandis the proofs of \cite[Theorem 6.7, Corollary 6.8]{CM21}.
\end{proof}

   We would like to point out a stronger form of the following theorem, which is proved by Maria Gorelik in a private communication. 

 \begin{thm}[Gorelik] \label{thm::16::Gorelik}
  Let $\mf s =\mf s^{-1}\oplus \mf s^0\oplus \mf s^{1}$ be an arbitrary finite-dimensional Lie superalgebra of  type I. Let $\psi: U(\mf s)\rightarrow U(\mf s^0)$  be the projection  to the first summand in the  decomposition  $U(\mf s) = U(\mf s^0)\oplus\left(\mf s^{-1} U(\mf s)+U(\mf s)\mf s^1\right).$ Let $m:=\dim \mf s^{-1}$ and $S:= \psi(\Lambda^{m} (\mf s^{1})  \Lambda^m (\mf s^{-1})).$ 
  Then, $$K(V) \text{ is simple }\Leftrightarrow SV \neq 0,$$
  for any simple $\mf s$-module $V$. 
 \end{thm} The Theorem \ref{thm::16::Gorelik} above implies that the simplicity of $K(V)$ depends only on the  annihilator of $V$ and leads to some interesting consequences. First, this applies to the more general setup of {\em map superalgebras} as studied in  \cite{Sa14, CM19, CM23}, that is, they are Lie superalgebras of the form  $\wtg[A] := \wtg\otimes A$ with the Lie bracket defined in the same fashion, for  finite-dimensional type I Lie superalgebras $\wtg$ and associative commutative algebras $A$. We may observe that Theorem \ref{thm::16::Gorelik}    can be applied to provide a more general   framework for general criteria of the simplicity of Kac modules for the map superalgebras $\wtg[A]$.

 Another interesting consequence is to give a complete description of the simplicity of Kac modules over Lie superalgebras of type I from Kac's list.  Such a characterization is known for the  Lie superalgebras $\mf{gl}(m|n)$ and $\mf{osp}(2|2n)$; see \mbox{\cite[Corollary 6.8]{CM21}}.  Now, we consider   the periplectic Lie superalgebra $\mf s := \pn$. It follows by an unpublished result of Vera Serganova  that    \begin{align} &K(\la) \text{ is simple } \Leftrightarrow \prod_{\alpha\in \Phi^+_\oa}(\langle \la+\rho,\alpha\rangle -1)\neq 0,\label{eq::pntypicalSer}
    \end{align}
   for each $\la \in \h^\ast$; see also \cite[Proposition 4.1]{CP24}. Suppose that $V$ is a simple module over $\mf s_\oa \cong \gl(n)$. Then, there exists a weight $\la \in \h^\ast$ such that the annihilators of $V$ is the same as that of $L(\la)$, by Duflo's theorem. It  follows by Theorem \ref{thm::16::Gorelik} that $K(V)$ is simple if and only if the conditions in \eqref{eq::pntypicalSer} hold. 
	\begin{rem}
	 Let $\g=\mf{sl}(2) =\langle e,h,f\rangle$ with the standard Chevalley generators $e,h$ and $f$, i.e., we have $[h,e]=2e$, $[h,f]=-2f$ and $[e,f]=h$. Consider the first Takiff algebra $\g_1$.	It is worth pointing out that not every central character of $\g_1$ is of the form $\chi_\la$, for $\la \in \h^\ast_1$. To see this, we set  the following notations 
		\begin{align*}
		&\ov e:=e^{(1)}= e\otimes \theta,~ \ov h:= h^{(1)}=h\otimes \theta,~\text{ and }\ov f:= f^{(1)}:=f\otimes \theta.
		\end{align*} 		 By \cite[Section 2.2]{Ch23} (see also \cite{M21}), one can show that the center $Z(\g_1)$ is  generated by the following two algebraically independent generators  $
		r_1:= \ov h^2+4\ov f\ov e,~r_2:= h\ov h+2\ov h+2(f\ov e+\ov fe). $ Define a central character $\chi: Z(\g_1)\rightarrow \C$ by letting $\chi(r_1)=0,~\chi(r_2)=1$. For any $\la \in \h_1^\ast$, we may observe that  
	\begin{align*}
	&\chi_\la(r_1)= \la(\ov h)^2,~\chi_\la(r_2) = \la(\ov h) (\la(h)+2).
	\end{align*} Therefore, $\chi$ does not equal the central character of any highest weight module.
\end{rem} 
 
 \section{Proofs of Theorems \ref{thm::1st} and \ref{thm::2nd}} \label{sect::eqv}
In this section, we let $\wtg=\gl(m|n), \mf{osp}(2|2n),$ or $\pn$ with a fixed  distinguished triangular decomposition. We shall keep the notations 
and assumptions of the previous subsections.
\subsection{Typical weights and typical Jordan blocks} \label{sect::31typ}  Recall from Subsection \ref{sect::132} that a Jordan block  is called typical provided that it contains a simple module with a typical highest weight.   The following is a restatement of Theorem \ref{thm::2nd}. 
\begin{thm} \label{thm::3}
  Suppose that $\wtg=\gl(m|n)$ or $\mf{osp}(2|2n)$ with $\kappa\in (\hellgeq)^\ast$. Suppose that $\wto_\ell^\kappa$ is a typical Jordan block in $\wto_\ell$. 
Then we  have an equivalence of categories: 
\begin{align*}
&K(\_):\mc O_\ell^\kappa \xrightarrow{\cong} \wto_\ell^\kappa, 
\end{align*} with the quasi-inverse $(\_)^{\wtg_\ell^1}$, preserving highest weights of any highest weight modules.
\end{thm}

We remark that there is an equivalence of categories for typical modules of the periplectic Lie superalgebra $\pn$ established by Serganova \cite[Theorem 5.2]{Se02}. It would be natural to establish an analog for general Takiff superalgebras associated with $\pn$.

The remaining part of this section is dedicated to the proof of Theorem \ref{thm::3}.

\subsection{Simplicity of highest weight Kac modules}\label{sect::322}
Before giving a proof of Theorem  \ref{thm::3}, we need some preparatory lemmas. In this subsection, we fix a  typical Jordan block  $\wto_\ell^\kappa$, for some $\kappa\in (\hellgeq)^\ast$.  For any element $x\in \wtg$ and subspace $\mf s\subseteq \wtg$,  recall from Subsection \ref{sect::22} that \begin{align*}
&x^{(i)} = x\otimes\theta^i\in \wtg_\ell,\text{ and }\mf s^{(i)} = \mf s\otimes \theta^i\subset  \wtg_\ell, \text{ for $0\leq i\leq \ell$.}
\end{align*} In addition, we set $x^{(i)}$ and $\mf s^{(i)}$ to be zero, for  $i>\ell$ or $i<0$.

For each $\la \in \h^\ast_\ell$, we define $K({\la}):= K(L(\la))$. As is well known,  the typical weights for Lie superalgebras $\mf{gl}(m|n), \mf{osp}(2|2n)$ and $\pn$ correspond to  the highest weights of simple Kac modules; see, e.g., \cite[Theorem 1]{Ka78}, \cite{PS89} and \cite[Proposition 4.1]{CP24}. The following is an analogue for Takiff superalgebras $\wtg_\ell$ of type I, for $\ell\geq 1$.
\begin{prop}\label{lem::4}
 Suppose that $\wtg = \gl(m|n), \mf{osp}(2|2n),$ or $\pn$ with $\la\in \h^\ast_\ell$. Then we have 
 $$K(\la)\cong \wtL({\la})\Leftrightarrow\text{$\la$ is typical}.$$ 
\end{prop}
\begin{proof} First, suppose that $\la\in\h^\ast_\ell$ is typical. Let $v_\la\in \wtL(\la)$ be a highest weight vector with respect to the distinguished Borel subalgebra $\widetilde{\mf b}$. Recall the reverse Borel subalgebra $\widetilde{\mf b}^r=\widetilde{\mf b}_\oa\oplus \wtg^{-1}$ and choose $X^{-}\in \Lambda^{\texttt{max}}(\wtg_\ell^{-1})$ to
 be a non-zero vector.  Recall from Subsections \ref{sect::oddreflbasic} and \ref{sect::oddrefl} that there is a chain of odd reflections and inclusions connecting Borel subalgebras from $\widetilde{\mf b}^r$ to $\widetilde{\mf b}$. Consequently,  $X^-v_\la$ is a $\widetilde{\mf b}^r$-highest weight vector of $\wtL(\la)$  by Proposition \ref{thm::Oddrefl}. 
 Recall that  $\wtL(\la)$ has a simple socle generated by $\Lambda^{\texttt{max}}(\wtg^{-1})\otimes L(\la)$. Let $\widetilde{v}_\la\in K(\la)$ be a $\widetilde{\mf b}$-highest weight vector of $K(\la)$  and $f: K(\la)\rightarrow \wtL(\la)$ be the canonical epimorphism such that $f(\widetilde{v}_\la) =v_\la$. Suppose on the contrary  that $K(\la)$ is not simple with a non-trivial proper submodule $M\subset K(\la)$, then  $X^-\widetilde{v}_\la \subseteq M$. This implies that $X^-v_\la=f(X^-\widetilde v_\la)=0$, which is contradictory. 

 Next, suppose that $K(\la)$ is a simple module for some weight $\la\in\h^\ast_\ell$. Subsequently, the $\widetilde{\mf b}^r$-highest weight vector of $K(\la)$ is $X^-v_\la$. It follows from Propositions \ref{thm::Oddrefl}  and \ref{thm::Oddreflforpn} that $\langle \la^{(\ell)}, \alpha\rangle \neq 0$, for all $\alpha\in \Phi_\ob^+$.  This completes the proof.
\end{proof}

Define $\psi_{\g_\ell}: U(\wtg_\ell)\rightarrow U(\h_\ell)$ to be the projection to the first summand in  the  decomposition $U(\wtg_\ell) = U(\g_\ell)\oplus\left(\wtg_\ell^{-1} U(\wtg_\ell)+U(\wtg_\ell)\wtg_\ell^1\right).$ Then the restriction of $\psi_{\g_\ell}$ to $Z(\g_\ell)$, also denoted by $\psi_{\g_\ell}$, defines a homomorphism of algebras. Recall the central elements $\widetilde{\Omega}\in Z(\wtg_\ell)$ and $\Omega'\in Z(\g_\ell)$ from Subsections \ref{sect::24typical} and \eqref{sect::262Kac}, respectively. We define  
$$\Omega:=\psi_{\g_\ell}(\widetilde{\Omega}).$$ Recall the Harish-Chandra homomorphism $\mf{hc}_\ell:Z(\g_\ell)\rightarrow U(\h_\ell)$ from Subsection \ref{sect::252}. The following corollary describes the relation between the two central elements. 
\begin{cor} \label{Coro::19::cent} Suppose that $\wtg=\mf{gl}(m|n)$ or $\mf{osp}(2|2n)$. Then, there is a positive integer $k$ such that $\Omega^k=c \Omega'$, for some non-zero scalar $c$.
\end{cor}
\begin{proof} For any weight $\la =(\la^{(0)},\ldots, \la^{(\ell)}) \in  \h^\ast_\ell$,  we have $\chi_\la(\Omega) = \la(\psi_{\h_\ell}(\Omega)) = \Pi_{\alpha\in \Phi^+_\ob}\langle \la^{(\ell)},\alpha\rangle$, by construction.  
Furthermore, it follows by Theorem \ref{prop::Kacsimple} and Proposition \ref{lem::4} that 
 $$\la \text{ is typical }\Leftrightarrow K(\la) \text{ is simple } \Leftrightarrow \la(\psi_{\h_\ell}(\Omega'))=\chi_\la(\Omega')\neq 0,$$ which  implies that $\psi_{\h_\ell}(\Omega')$ is divided by $\psi_{\h_\ell}(\Omega)$ in $U(\h_\ell).$ Since $\psi_{\h_\ell}(\Omega')$ is invariant under the dot-action of $W$ on $\h_\ell^\ast$ from \eqref{eq::dotaction} by Lemma \ref{lem11::G}, it follows that  $c\psi_{\h_\ell}(\Omega') = \psi_{\h_\ell}(\Omega)^k$, for some non-zero scalar $c\in \C$ and a  positive integer $k$.  
The conclusion follows from the injectivity of  $\psi_{\h_\ell}$ by Lemma \ref{lem11::G} again.
\end{proof}

We are now in a position to give a proof of Theorem \ref{thm::1st}. 
\begin{proof}[Proof of Theorem \ref{thm::1st}]
 The assertions in Part (i) follows by  the construction of $\widetilde{\Omega}$ from Subsection \ref{sect::24typical}. Next, suppose that $\wtg=\gl(m|n)$ or $ \mf{osp}(2|2n)$. Let $L$ be a simple module over $\wtg_\ell$ with $\widetilde{\chi}$ as its  central character. Then there is a simple $\g_\ell$-module $V$ such that $L$ is a quotient of $K(V)$ by \cite[Theorem A]{CM21}. We may observe that 
 \begin{align*}
  &\widetilde{\chi}(\widetilde{\Omega})\neq 0 \\  
  & \Leftrightarrow   \Omega'V\neq 0, \text{ by Corollary \ref{Coro::19::cent}} \\
   & \Leftrightarrow  K(V) \cong L \text{ by Theorem \ref{prop::Kacsimple}.}
 \end{align*} Furthermore, if $\widetilde{\chi} = \widetilde{\chi}_\la$, for some $\la \in \h^\ast_\ell$. Then $\widetilde{\chi}_\la$ is typical if and only if $K(\la)$ is simple, which is equivalent to the condition that $\la$ is typical by Proposition \ref{lem::4}. This proves the assertion in  Part (ii). The assertion in Part (iii) follows from \mbox{Proposition \ref{lem::4}}.
\end{proof}

\subsection{Irreducible characters of finite-dimensional modules over Takiff superalgebras} \label{sect::ch33}  In this subsection, we assume that $\widetilde{\g} =\mf{gl}(m|n), \mf{osp}(2|2n)$ or $\pn$. In \cite{BR13}, Babichenko and Ridout give a description of simple  modules over the $1$-st Takiff superalgebra associated with $\gl(1|1)$.  In this subsection, we will study the character problem of  finite-dimensional simple modules over type I Takiff superalgebras. 
  \begin{lem}
   Suppose that $\wtg = \gl(m|n), \mf{osp}(2|2n),$ or $\pn$. Let $\la\in \h^\ast_\ell$. Then the following are equivalent:
   \begin{itemize}
       \item[(1)] $\widetilde{L}(\la)$ is finite-dimensional.
       \item[(2)]  Let $\alpha\in \Phi^+_\oa$. Then  $\langle \la^{(i)}, \alpha \rangle$ is a non-negative integer  if $i=0$, and it is zero otherwise.
   \end{itemize}
  \end{lem}
  \begin{proof} Let $\g = [\g,\g]\oplus\mf z(\g)$, where $\mf z(\g)$ is the center of $\g$.
  Using the Kac induction functor, we can show that $\dim \widetilde{L}(\la)<\infty$ if and only if $\dim L(\la)<\infty$.  Assume the condition $(2)$,  then $L(\la)$ is lifted from a simple highest weight $\g$-module on which  $[\g,\g]\otimes \theta^k$ acts trivially and $\mf z(\g)\otimes  \theta^k$ acts as a scalar, for each $k>0$. It follows by the standard representation theory that $\dim L(\la)<\infty$. 
  
  Next, we assume that  $\dim L(\la)<\infty$.     Suppose on the contrary that $\langle \la^{(k)}, \alpha \rangle \neq 0$, for some $k>0$ and $\alpha\in \Phi_\oa^+$. We may choose $k$ to be the smallest integer satisfying this condition, and so $L(\la)$ can be regarded as the simple highest weight module over the Takiff Lie superalgebra $\mf g_k$ with the highest weight $(\la^{(0)}, \ldots, \la^{(k)}) \in \h^\ast_k$. Now,   let $v\in L(\la)$ be a highest weight vector. We can choose a subalgebra   $\langle e,h,f\rangle\subset \g$  consisting of an $\mf{sl}(2)$-triple  such that $e\in \wtg^\alpha, f\in \wtg^{-\alpha}$ and $h\in \mf h_\ell$. We then form the subalgebra $\mf s_k\subset \mf g_k$. Since  $U(\mf s_k)v$ is a non-trivial  image of a simple Verma module over $\mf s_k$ by \cite[Theorem 7.1]{Wi11}, we may conclude that $L(\la)$ is infinite-dimensional. Therefore, $L(\la)$ is lifted from a finite-dimensional $\g$-module. This completes the proof.
  \end{proof}

 For each element $w$ in the Weyl group $W$ of $\g$, we let $\texttt{len}(w)$ be the length of $w$. Let $\la =(\la^{(0)},\la^{(1)},\ldots, \la^{(\ell)}) \in \h^\ast_\ell$. Recall that $\rho$ and $\widetilde{\rho}$ denote the Weyl vectors of $\g$ and $\wtg$, respectively. We formulate the following definitions of {\em semitypical weights}, generalizing the definition in \cite{BR13} for Takiff $\mf{gl}(1|1)$: 
\begin{enumerate}
    \item[(i)] Suppose that $\widetilde{\g} =\mf{gl}(m|n), \mf{osp}(2|2n)$.    Then $\la$ is  called semitypical if
  \[\text{either $\prod_{\alpha \in \ov{\Phi}_{\ob}^+}\langle \la^{(0)} +\widetilde\rho,  \alpha\rangle\neq 0$ or $\prod_{\alpha \in \ov{\Phi}_{\ob}^+}\langle \la^{(i)}, \alpha\rangle\neq 0$ for some $i>0.$}\]  
  \item[(ii)]  Suppose that $\widetilde{\g} =\pn$.  Then $\la$  is called  semitypical if 
  \[\text{either $\prod_{\alpha \in {\Phi}_{\oa}^+}\langle \la^{(0)} +\rho,  \alpha\rangle\neq 1$ or $\prod_{\alpha \in {\Phi}_{\oa}^+}\langle \la^{(i)}, \alpha\rangle\neq 0$ for some $i>0.$}\]  
\end{enumerate}

  The following theorem, generalizing the classical formula of Kac \mbox{\cite[Proposition 2.8]{Ka78}}, is the main result in this subsection: 

 \begin{thm} \label{thm24::characters} Suppose that $\widetilde{\g} =\mf{gl}(m|n), \mf{osp}(2|2n)$ or $\pn$.
 Let $\la =(\la^{(0)},\ldots, \la^{(\ell)}) \in \h^\ast_\ell$ be   such that $\dim \wtL(\la)<\infty$. 
Then we have 
\begin{enumerate}
    \item[(1)]   Suppose that  $\la$ is semitypical and $0\leq k\leq \ell$ is the smallest number such that $\prod_{\alpha \in \Phi_{\ob}^-}\langle \la^{(i)}, ~ \alpha\rangle= 0,$ for all $k<i$. Then
  \begin{align}
  &\ch \widetilde L(\la) = \frac{\prod_{\alpha\in \Phi_\ob^-} (1+e^\alpha)^{k+1}}{\prod_{\beta\in \Phi_\oa^-} (1-e^{\beta})}\sum_{w\in W} (-1)^{\emph{\texttt{len}(w)}} e^{w(\la^{(0)}+\rho)-\rho} \label{eq::ChThm21}
  \end{align}
    \item[(2)] Suppose that $\la$ is not semitypical. Then $\ch \wtL(\la)$ coincides with the character of the simple highest weight module over $\wtg$ with highest weight $\la^{(0)}$. 
\end{enumerate}

 \end{thm}
 \begin{proof} First, assume that $\la$ is semitypical. We may note that $$\prod_{\beta\in \Phi_\oa^-}\frac{1}{1-e^{\beta}}\sum_{w\in W} (-1)^{\texttt{len}(w)} e^{w(\la^{(0)}+\rho)-\rho}$$ is the character of the finite-dimensional simple module over $\g$ with the highest weight $\la^{(0)}$.  We will proceed with an induction on the non-negative integer $\ell -k$. If it is zero, then the assertion can be directly derived from Proposition \ref{lem::4}. Suppose that $\ell>k$, then we have  $\prod_{\alpha \in \Phi_{\ob}^-}\langle \la^{(i)}, ~ \alpha\rangle= 0,$ for all $k<i$. In particular, this implies that $\la^{(\ell)}$ is the weight of a one-dimensional module over $\wtg$ and allows us to define a one-dimensional module $V_\la$ over $\wtg_\ell$ such that the weight of $V_\la$ is  $(0, 0, \ldots, -\la^{(\ell)})\in \h^\ast_\ell$. We may note that the functor $\_\otimes V_\la: \wto_\ell\rightarrow \wto_\ell$ forms an auto-equivalence of $\widetilde{\mc O}_\ell$ preserving characters. Let $\mu = (\mu^{(0)},\ldots, \mu^{(\ell)})\in \h^\ast_\ell$ be the  highest weight of $\wtL(\la)\otimes V_\la$. We may note that $\mu^{(\ell)}=0$ and so $\wtL(\la)\otimes V_\la$ can be regarded as a simple highest weight module over $\wtg_{\ell-1}$ with highest weight $(\la^{(0)}, \la^{(1)},\ldots, \la^{(\ell-1)})$. The conclusion follows  by induction hypothesis. 

  Finally, we assume that  $\la$ is not semitypical. Then the module $\widetilde L(\la)$ can be regarded as a simple highest weight module over $\wtg$ with the highest weight $\la^{(0)}$ by letting $\wtg\otimes \theta^i$ act  trivially on $\widetilde L(\la)$, for all $i>0$.  This completes the proof.
 \end{proof}

\begin{rem} To describe the characters of irreducible finite-dimensional modules $\wtL(\la)$ with $\la\in \h^\ast_\ell$ non-semitypical, we can reduce to  the problem of determining the characters of finite-dimensional irreducible modules over Lie superalgebras of type $\mf{gl}, \mf{osp}$ or $\mf p$, as  has already been established in \cite{Br03, CL10, CLW11, B+9}.
\end{rem}

 	\begin{rem} \label{rem26::gl11}Suppose that $\wtg= \gl(1|1)$. Then all simple $\wtg_\ell$ are finite-dimensional, since the even part of $\wtg_\ell$ is abelian.  Let $E:=E_{12}, F:=E_{21}$ and $H:=E_{11}+E_{22}\in \gl(1|1)$.  Let $\la \in \h^\ast_\ell =(\la^{(0)}, \ldots,\la^{(\ell)})$ be a weight. If $\la$ is not semitypical, then $\wtL(\la)$ is lifted from a simple $\gl(1|1)$-module.   If $\la$ is semitypical, there is $0\leq k\leq \ell$   such that $\la^{(k)}(H)\neq 0$ and $\la^{(i)}(H)=0$, for all $i>k$.  Let $\alpha$ be the unique positive odd root of $\gl(1|1)$, that is, $\alpha =\vare_1-\vare_2$, which we regard as an element in $\h^\ast_\ell.$  Let $v$ denote a highest weight vector of $K(\la +(\ell-k)\alpha)$. Then the following set  	\begin{align*}
		&{\bf B}:= \{F^{(i_1)}F^{(i_2)}\cdots F^{(i_k)} \cdot  F^{(q_\la+1)}\cdots  F^{(\ell)} v|~0\leq i_1<i_2<\cdots <i_k\leq k\}
	\end{align*} spans a submodule isomorphic to $\wtL(\la)$. 
    \end{rem}

\subsection{Proof of Theorem \ref{thm::3}}\label{sect::34}
The goal of this subsection is to complete the proof of \mbox{Theorem \ref{thm::3}}. We fix an element $\kappa\in (\hellgeq)^\ast$ such that $\wto_\ell^\kappa$ is a typical Jordan block, throughout this subsection. Recall the equivalence relation $\sim$ on $\h^\ast_\ell$, as defined in Subsection \ref{sect::223}.

 The following lemma will be useful. 
\begin{lem} \label{coro::6}
 Let $M\in \wto_\ell^\kappa$ be indecomposable such that  the multiplicities $[M: \wtL(\la)]$ and $[M:\wtL(\mu)]$ are non-zero, for some $\la,\mu\in \h^\ast_\ell$. 
Then $\la^{(0)} -\mu^{(0)} \in \Z\Phi_\oa$.  
Furthermore, $M$ is generated by its $\wtg_\ell^1$-invariant subspace, that is, $U(\wtg_\ell)\cdot M^{\wtg_\ell^1}=M$.
\end{lem}
\begin{proof}
	Define a relation $\rightarrow_0$ on $\h^\ast\times \{\kappa\}\subseteq \h^\ast_\ell$ by declaring that 
	\begin{align*}
		&\gamma\rightarrow_0 \eta, \text{ if }[{M}(\gamma):L(\eta)] \neq 0.
	\end{align*} Let $\sim_0$ be the equivalence relation on $\h^\ast\times \{\kappa\}$ generated by $\sim_0$. 
     	By    Proposition \ref{lem::4}, it follows that    
     		$\gamma\rightarrow \eta $ if and only if $ \gamma\rightarrow_0 \eta,$ for any   weights $\gamma,\eta\in \h^\ast\times \{\kappa\}$. Therefore, $\gamma\sim \eta$ if and only if $\gamma\sim_0\eta$. The first assertion now follows by \mbox{Proposition \ref{lem::blockdec}}.
     
     Next, we prove the second assertion.  By Lemma \ref{lem::1} there exists a filtration 
     \begin{align*}
     &0\subseteq M_0\subseteq M_1\subseteq M_2\subseteq \cdots \subseteq M_{k+1}=M,
     \end{align*} in which $M_{i+1}/M_i$ are highest weight modules of highest weights $\la_i\in \h^\ast_\ell$, for $0\leq i\leq k$. Let $v_i\in M_{i+1}$ be the pre-image of a  highest weight vector in $M_{i+1}/M_i$, under the canonical quotient map $M_{i+1}\rightarrow M_{i+1}/M_i$. We may note that $v_0,v_1,\ldots, v_{k}$ generate $M$. Therefore, it suffices to show that  all the vectors $v_0,v_1,v_2,\ldots ,v_{k}$ lie in $M^{\wtg_\ell^1}$. Suppose on the contrary that $\wtg^1_\ell\cdot v_j\neq 0$, for some $0\leq j\leq k$. Then the weight subspace $M^{\la_j+\alpha}$ is non-zero, for some root $\alpha\in \Phi^+_\ob$.  Note in addition that there is $0\leq i\leq k$ such that   $\la_j+\alpha  \in \la_i-\Z \Phi^+$.  However, by the first assertion, we have $\la_i-\la_j\in \Z\Phi_\oa$, a contradiction. This completes the proof. 
\end{proof}

Let $\Id_{\mc O_\ell^\kappa}$ and $\Id_{\wto_\ell^\kappa}$ denote the identity (endo)functors on ${\mc O_\ell^\kappa}$ and ${\wto_\ell^\kappa}$, respectively. 
Let $\phi_{(\_)}: \Id_{\mc O_\ell^\kappa} \rightarrow K(\_)^{\wtg_\ell^1} $ and $\psi_{(\_)}: K(\_^{\wtg_\ell^1}) \rightarrow \Id_{\wto_\ell^\kappa}$ be the counit and unit adjunctions, respectively. For any $V\in \mc O_\ell^\kappa$ and $M\in \wto_\ell^\kappa$, we have 
\begin{align*}
&\phi_V: V \rightarrow K(V)^{\wtg_\ell^1},~v\mapsto 1\otimes v,~\text{for any }v\in V;\\
 &\psi_M:   K(M^{\wtg_\ell^1})\rightarrow M,~u\otimes m\mapsto um,~\text{for any }u\in U(\wtg_\ell),~m\in M.
\end{align*}
\begin{lem} \label{lem::77} If  $V\in \mc O_\ell^\kappa$ and $M\in \wto_\ell^\kappa$ have finite lengths, then 
	both $\phi_V$ and $\psi_M$ are isomorphisms.
\end{lem}
\begin{proof} 
 We first show that $\phi_V$ is an isomorphism. Let $0\rightarrow V_0\rightarrow V\xrightarrow{f} V_1\rightarrow 0$ be a short exact sequence in $\mc O_\ell^\kappa$ such that $V_1$ is a simple module. By the exactness of $K(\_)$, we get a short exact sequence in $\wto_\ell$:
	\begin{align*}
	&0\rightarrow K(V_0)\rightarrow K(V) \xrightarrow{K(f)} K(V_1)\rightarrow 0. 
	\end{align*}  By Proposition \ref{lem::4}, $K(V_1)$ is simple. It follows by \mbox{\cite[Remark
3.3, Corollary 4.3]{CM21}} that  $K(V_1)^{\wtg_\ell^1}\cong V_1$ (see also  \cite[Corollary 4.5]{CM21}). Now, consider the homomorphism $(K(f))^{\wtg_\ell^1}: K(V)^{\wtg_\ell^{1}}\rightarrow K(V_1)^{\wtg_\ell^1}$ induced by applying $(\_)^{\wtg_\ell^1}$ to $K(f)$.  Since   $1\otimes V\subseteq K(V)^{\wtg_\ell^{1}}$  is sent to $V_1$ under $(K(f))^{\wtg_\ell^1}$,  the homomorphism  $(K(f))^{\wtg_\ell^1}$ is  non-zero. Therefore, we obtain a commutative diagram with exact rows and columns 
	\begin{align}
	&\xymatrixcolsep{2pc} \xymatrix{
		0 \ar[r]     &V_0 \ar[r]  \ar@<-2pt>[d]^{\phi_{V_0}}   &  V\ar[d]^{\phi_{V}} \ar[r]  &V_1 \ar[r] \ar@<-2pt>[d]^{\phi_{V_1}}  &0 \\0 \ar[r]    &   K(V_0)^{\wtg_\ell^1} \ar[r] &   K(V)^{\wtg_\ell^1} \ar[r] & K(V_1)^{\wtg_\ell^1} \ar[r] & 0} \label{eq::8}
\end{align} By using induction on the lengths of modules and the short five lemma, we can deduce that $\phi_V$ is an isomorphism. 

Next, we will show that $\psi_M$ is also an isomorphism. Let $C$ and $D$ be the kernel  and cokernel of $\psi_M$, respectively, and so that there is the following   exact sequence
\begin{align*}
&0\rightarrow C\rightarrow K(M^{\wtg_\ell^1}) \xrightarrow{\psi_M} M \rightarrow D\rightarrow 0.
\end{align*}  By Lemma \ref{coro::6}, we have $D=0$. 
It remains to show that $C=0$. To see this, we may observe that  $\C 1\otimes M^{\wtg_\ell^1}\subseteq K(M^{\wtg_\ell^1})^{\wtg_\ell^1}$ surjects onto $M^{\wtg_\ell^1}$ under the homomorphism $(\psi_M)^{\wtg_\ell^1}$. This implies that  $K(M^{\wtg_\ell^1})^{\wtg_\ell^1}\xrightarrow{(\psi_M)^{\wtg_\ell^1}} M^{\wtg_\ell^1}$ is an epimorphism. Now, we get a short exact sequence
\begin{align*}
&0\rightarrow C^{\wtg_\ell^1}\rightarrow K(M^{\wtg_\ell^1})^{\wtg_\ell^1} \rightarrow M^{\wtg_\ell^1}\rightarrow 0.
\end{align*} As we have already proven above,  $\phi_{M^{\wtg_\ell^1}}:  M^{\wtg_\ell^1} \xrightarrow{\cong} K(M^{\wtg_\ell^1})^{\wtg_\ell^1}$ is an isomorphism, it follows that $C^{\wtg_\ell^1}=0$. Since $C$ admits a filtration of highest weight modules,  we can conclude that $C=0$. This completes the proof. 
\end{proof}

Now we are in a position to prove Theorem \ref{thm::3}.
\begin{proof}[Proof of Theorem \ref{thm::3}]
 To complete the proof, it suffices to show the following claim:
	
	{\bf Claim.} {\em Let $V\in \mc O_\ell^\kappa$ and $M\in \wto_\ell^\kappa$. Then
	both $\phi_V$ and $\psi_M$ are isomorphisms.}
	
	To prove the claim, we only need to prove that $\phi_V$ is an isomorphism, since the assertion about the  isomorphism of $\psi_M$ can be proved by an argument identical to that used in Lemma \ref{lem::77}. 
	
	We first note that $\phi_V$ is an inclusion. It remains to show the surjectivity of $\phi_V: V\rightarrow K(V)^{\wtg_\ell^1}$. Suppose that, on the contrary, there is $k>0$ such that $$(\Lambda^k(\wtg_\ell^{-1})\otimes V)^{\wtg_\ell^1}\neq 0.$$
	Then there is a non-zero vector $\sum_{i=1}^s y_i\otimes v_i\in (\Lambda^k(\wtg_\ell^{-1})\otimes V)^{\wtg_\ell^1}$, in which we assume that $y_1,\ldots,y_s\in \Lambda^k(\wtg_\ell^{-1})$ are linearly independent and $v_1,\ldots,v_s \in V$ are non-zero vectors. 
	 By \cite[Lemma 6.2]{CT23}, 
	there exists a descending filtration 
	\begin{align*}
		V=V_0\supseteq  V_1\supseteq V_2\supseteq \cdots,  
	\end{align*} such that $\bigcap_{i\geq 0} V_i =0$ and $V_j/V_{j+1}$ are simple, for all $j\geq 0$. 
	  Then there is $p\geq 0$ such that $v_1,v_2,\ldots ,v_s \in V\backslash V_p$. Let $\ov{v_i}$ denote the image of $v_i$ under the quotient map $V\rightarrow V/V_p$, for $1\leq i\leq s$. However, we observe  that $\sum_{i=1}^s y_i\otimes \ov{v_i}$ is a non-zero vector in $(\Lambda(\wtg_\ell^{-1})^k\otimes V/V_p)^{\wtg_\ell^1}\subseteq K(V/V_p)^{\wtg_\ell^1}$, and it does not lie in $\phi_{V/V_p}(V/V_p) =V/V_p \subseteq K(V/V_p)^{\wtg^1_\ell}$. Since $V/V_p$ is of finite length, this contradicts the conclusion in \mbox{Lemma \ref{lem::77}.} This completes the proof.
\end{proof}

 \section{Whittaker modules} \label{sect::4}
 Throughout this section, we let $\wtg=\gl(m|n), \mf{osp}(2|2n)$, or $\pn$. We shall retain the notations and assumptions of  the previous  sections. 
Recall that we fix  $\wtg =\widetilde{\mf n}\oplus \mf h\oplus \widetilde{\mf n}^-$ to be the   triangular decomposition from Section \ref{sect::2.1}.  
 \subsection{Non-singular Whittaker modules}
 \label{sect::42} We define $\mf n^{\pm}:=(\widetilde{\mf n}^\pm)_\oa$ and the corresponding Takiff superalgebras $\mf n_\ell^\pm$.  
A character $\zeta: {\mf n}_{\ell}\rightarrow \C$  is called \textit{non-singular} if $\zeta(\mathfrak{g}_{\ell}^{\alpha})\ne0$, for  any simple root $\alpha$. Here $\g^\alpha_\ell$  denotes the root space of $\g_\ell$ corresponding to $\alpha$.

Throughout this section, we assume that $\zeta: {\mf n}_{\ell}\rightarrow \C$ is a non-singular character. A $\g_\ell$-module $M$ is said to be a Whittaker $\g_\ell$-module associated to $\zeta$ if $M$ is finitely generated, locally finite over $Z(\g_\ell)$ and $x-\zeta(x)$ acts on $M$ locally nilpotently, for any $x\in \mf n_\ell$.  
Let  $\mathcal{N}_\ell(\zeta)$  denote the category of all Whittaker $\g_\ell$-modules associated to $\zeta$. The following lemma shows that such a category is independent of the choice of the non-singular characters $\zeta$, up to equivalence. 

\begin{lem}\label{lem::200}
Suppose that $\zeta, \eta :\mf n\rightarrow \C$ are two non-singular characters. Then $\mathcal{N}_\ell(\zeta)$ and $\mathcal{N}_\ell(\eta)$ are equivalent. 
\end{lem} 
\begin{proof} We note that an automorphism   $\phi: \g\rightarrow \g$ can be extended to an automorphism $\phi_\ell$ of $\g_\ell$ by letting $\phi_\ell(x\otimes \theta^i) = \phi(x)\otimes \theta^i$, for $0\leq i\leq \ell$.  Using this fact, the proof can be completed by an argument which is completely analogous to the one given in the proof of  \cite[Proposition 36]{CC23},  thus we give a sketch of the proof and omit the details. 
	
	Let $\{e_1,e_2,\ldots,e_p\}\subset \mf n_\ell$ be the set of all root vectors associate to simple roots. Without loss of generality, we can assume that $\eta(e_i) =1$, for $1\leq i\leq p$. Then there is an automorphism $\phi$ of $\g_\ell$ such that $\phi(e_i) = \zeta(e_i)^{-1}e_i$. For any $M\in  {\mc N}_\ell(\zeta)$, we define an associated $\g_\ell$-module  $T_{\zeta,\eta}(M)$, that is, $T_{\zeta,\eta}(M)$ has the same underlying subspace as $M$ with the new $\g_\ell$-module action denoted by the star-notation $\ast$, given as follows:
	\begin{align}
	&x\ast m :=\phi(x)m,
	\end{align}  for any $x\in \g_\ell$ and $m\in M$. Since $x-\zeta(x)$ acts locally nilpotently on $M$, for $x\in \mf n_\ell$, it follows that $x-1$ acts on $T_{\zeta,\eta}(M)$ locally nilpotently, for $x\in \mf n_\ell$. By definition, $T_{\zeta,\eta}(M)$ is  finitely generated over  $U(\g_\ell)$ and locally finite over $Z(\g_\ell)$. Consequently, this yields an equivalence $T_{\zeta,\eta}(\_)$  from $\mathcal{N}_\ell(\zeta)$ to $\mathcal{N}_\ell(\eta)$. This completes the proof.
\end{proof}

The following lemma is an analogue of \cite[Proposition 5]{CoM15} for Takiff algebras.
\begin{lem} \label{lem::21} There is an isomorphism of algebras
	$  {Z}(\mathfrak{g}_{\ell}) \otimes  {U}(\mathfrak{n}_{\ell}) \cong  {Z}(\mathfrak{g}_{\ell}) {U}(\mathfrak{n}_{\ell}).$
\end{lem}
\begin{proof}
	Let  
	$\phi:Z(\g_\ell)\otimes {U}(\mf {n}_\ell) \rightarrow \  Z(\g_\ell)U(\mf n_\ell)$ be the  canonical multiplication map, sending elements $\sum_{i=1}^{k}z_i\otimes u_i\in Z(\g_\ell)\otimes U(\mf n_\ell)$ to $\sum_{i=1}^{k}z_i u_i.$ It remains to  show that $\phi$ is injective. 
  
    For each $z\in Z(\g_\ell)$, note that $z=h_z+\sum_{i=1}^px^-_ih_ix^+_i,$ for some $h_z, h_1,\ldots, h_p\in U(\h_\ell)$ and $x_i^\pm \in {\mf n}^\pm_\ell U(\mf n^\pm_\ell)$.   
   Recall  the following homomorphism  from Subsection \ref{sect::252}: $$\psi=\psi_{\h_\ell}:\ Z(\g_\ell)\rightarrow  {U}(\mathfrak{h}_\ell),~z\mapsto h_z, \text{for }z\in Z(\g_\ell).$$      Now, let $\sum\limits_{i=1}^kz_i\otimes u_i\in {Z}(\mathfrak{g}_{\ell})\otimes {U}(\mathfrak{n}_{\ell})$ lie in the kernel of $\phi$ and $z_i's$ are linearly independent.
   By the PBW theorem, we have 
	$\sum\limits_{i=1}^k\psi(z_i)u_i=0\quad\text{in } {U}(\mathfrak{g_{\ell}}),$
	Since $\big\{\psi(z_i)\big\}_{i=1}^k$ are linearly independent, it follows that all  $\psi(z_i)u_i=0$, for all $i=1,\ldots k$, which implies that $u_i=0$ for all $i=1,\ldots,k$. This completes the proof.
\end{proof}

For characters $\chi:{Z}(\g_\ell)\rightarrow\mathbb{C}$ and $\zeta:\mathfrak{n}_{\ell}\rightarrow\C$, we define a one-dimensional  ${Z}(\mathfrak{g}_{\ell}){U}(\mathfrak{n}_{\ell})$-module $\mathbb{C}_{\chi,\zeta}$ by 
$$\big(x-\chi(x)\big)\mathbb{C}_{\chi,\zeta}=0=\big(y-\zeta(y)\big)\mathbb{C}_{\chi,\zeta},\ \text{for all } x\in{Z}(\mathfrak{g}_{\ell}), \ y\in{U}(\mathfrak{n}_{\ell}).$$

As a consequence of Lemma \ref{lem::21}, we provide the following corollary, which presents a complete classification of simple modules over $Z(\g_\ell)U(\mf n_\ell)$.
\begin{cor} \label{cor::24}
The set $\big\{\mathbb{C}_{\chi,\zeta}|\ \chi: {Z}(\mathfrak{g_\ell})\rightarrow\mathbb{C}\text{ and } \zeta:  \mathfrak{n}_{\ell}\rightarrow \C \text{ are  characters}\}$ is a complete 
set of representatives of the isomorphism classes of simple modules over  ${Z}(\mathfrak{g}_{\ell}){U}(\mathfrak{n}_{\ell}).$
\end{cor}

\begin{cor} \label{cor::27}	
Let $\zeta: \mf n_\ell\rightarrow \C$ be an arbitrary character. Then the category $\mc N_\ell(\zeta)$ is a Serre subcategory of $\g_\ell$$\mod$.
\end{cor}
\begin{proof} Let $R:=Z(\g_\ell)U(\mf n_\ell)$. 
By Lemma \ref{lem::21} and Corollary \ref{cor::24}, $R$ is a finitely generated subalgebra of $U(\g_\ell)$ such that all simple $R$-modules are one-dimensional. The conclusion follows by a standard argument. 
\end{proof}

	\subsection{Finite $W$-algebras associated with  Takiff  algebras} \label{sect::43}
	
	Finite $W$-algebras are 
	certain associative algebras, which can be realized through the Zhu algebras of the affine $W$-algebras constructed from pairs $(\g, e)$, where $\g$ is a complex semisimple Lie algebra and $e\in \g$ is a nilpotent element. Recently,  the  representation theory of finite $W$-algebras associated with reductive Lie algebras has received a great amount of attention, since Premet’s work \cite{Pr02}; see also  \cite{GG02,BGK08,Lo10, Lo11, Lo12, Pr07, Pr07Mo, Wa11} and reference therein. However, at present, analogous finite $W$-algebras for more general Takiff algebras have received less attention.

	Recently, in  \cite{He18, He22} He studied the finite $W$-algebra associated with generalized Takiff algebras $\g_\ell$, for reductive Lie algebras $\g$.	
	
	\subsubsection{} Let $e\in \g$ be a nilpotent element inside an $\mf{sl}(2)$-triple  $\langle e,h,f\rangle \subseteq \g$. Then the adjoint action of $h$ on $\g$ defines a  {\em Dynkin $\Z$-grading}
\begin{align}
&\g = \bigoplus_{k\in \Z} \g(k), \text{ where } \g(k):= \{x\in \g_\ell|~[h,x] = kx\}. \label{eq::D1} 
\end{align}	 
	This Dynkin grading is a good grading for $e$ in the sense \cite{EK05}.  By \mbox{\cite[Lemma 2.2.1]{He18}}, \eqref{eq::D1} induces a good $\Z$-grading on $\g_\ell$ as follows:
	\begin{align}
		&\g_\ell = \bigoplus_{k\in \Z} \g_\ell(k), \text{ where } \g_\ell(k):= \{x\in \g_\ell|~[h,x] = kx\} =\bigoplus_{i=0}^\ell \g(k)\otimes \theta^i. \label{eq::D2} 
	\end{align}

The nilpotent element $e$ is called  {\em principal}  in $\g$ provided that the kernel of the adjoint action of $e$ on $\g$ has minimal dimension, which is equal to the rank of $\g$. 
From now on, we choose a principal nilpotent element $e$ such that the grading \eqref{eq::D1} is even (i.e., $\g(2k+1)=0$, for all $k\in \Z$) and the nilpotent subalgebra $\bigoplus_{k\leq -2} \g(k)$ from \eqref{eq::D1} coincides with the even nilradical $\mf n$ from Subsection \ref{sect::2.1}. 
Fix a non-degenerate, invariant, symmetric, and bilinear form $(\_|\_): \g\times \g\rightarrow \C$ on $\g$. 	By  \mbox{\cite[Lemma 2.1.3]{He18}}, this induces a  non-degenerate, invariant, symmetric, and bilinear form $(\_|\_)_\ell$ on $\g_\ell$ determined by $$(\sum_{i=0}^\ell x_i\otimes\theta^i| \sum_{j=0}^\ell y_j\otimes\theta^j)_\ell =\sum_{k=0}^\ell \sum_{i+j=k}^\ell(x_i|y_j),\text{ for all }x_0,\ldots,x_\ell,y_1,\ldots,y_\ell\in \g.$$

Let   $\zeta^e(\_):\mf n_\ell \rightarrow \C$ be a linear map given by  \begin{align}
&\zeta^e(\sum_{j=0}^\ell y_j\otimes\theta^j) = (e|\sum_{j=0}^\ell y_j\otimes\theta^j)_\ell =\sum_{j=0}^\ell(e|y_j), \label{eq::chidef}
\end{align} for $y_0, y_1, \ldots, y_\ell\in \g$. By  \cite[Lemma 2.2.5]{He18} (see also \cite{Cas11}), $\zeta^e$ defines a non-singular character on $\mf n_\ell$.   Following \cite[Section 2.2]{He18}, we define the corresponding  finite $W$-algebra $U(\g_\ell,e)$ as follows:
\begin{align}
	&U(\g_\ell,e):=\End_{U(\g_\ell)}(Q_{\zeta^e})^{\text{op}},\label{def::generalW}
\end{align} where $Q_{\zeta^e}:=U(\g_\ell)/I_{\zeta^e}$ is the Gelfand-Graev type module with the left ideal $I_{\zeta^e}$ of $U(\g_\ell)$   generated by elements $x-\zeta^e(x)$, for $x\in \mf n_\ell$.  Let $\pr: U(\g_\ell)\rightarrow U(\g_\ell)/I_{\zeta^e}$ denote the natural projection. As in the classical case,  $U(\g_\ell, e)$ can be identified with the following associative algebra:
\begin{align}
&\{\pr(y)\in Q_\zeta|~[x,y]\in I_{\zeta^e}, \text{for all }x\in \mf{n}_\ell\},
\end{align}  with the multiplication given by $\pr(y_1)\pr(y_2) = \pr(y_1y_2),$ for any $y_1, y_2\in U(\g_\ell)$ such that $[x,y_1], [x,y_2]\in I_{\zeta^e}$ for all $x\in \mf{n}_\ell$. By \cite[Theorem 2.4.2]{He18}, $\pr(\_): U(\g_\ell,e)\rightarrow Z(\g_\ell)$ is an isomorphism, and therefore they are isomorphic to a polynomial algebra by \cite{RT92}.

Let $\g_\ell$-Wmod$^{\zeta^e}$  denote the category of all  $\g_\ell$-modules on which $x-\zeta^e(x)$ acts locally nilpotently, for all $x\in\mf n_\ell.$   We note that the category $\mc N_\ell(\zeta^e)$ of Whittaker modules associated to $\zeta^e$ is a full subcategory of  $\g_\ell$-Wmod$^{\zeta^e}$. 
 Also, let $U(\g_\ell, e)\Mod$  denote the category of all $U(\g_\ell, e)$-modules.
 For any $M\in  \g_\ell\text{-Wmod}^{\zeta^e}$, we define
\begin{align}
	&\text{Wh}_{\zeta^e}(M):=\{m\in M|~xm=\zeta^e(x)m,~\text{for any }x\in \mf n_\ell\}.
\end{align}
The elements of $\text{Wh}_{\zeta^e}(M)$ are called {\em Whittaker vectors} of $M$. By a standard argument (see also \cite[Lemma 2.4.7]{He18}), there are the following well-defined functors:
\begin{align}
&\text{Wh}_{\zeta^e}(\_):~ \g_\ell\text{-Wmod}^{\zeta^e} \rightarrow  U(\g_\ell, e)\Mod,~M\mapsto \text{Wh}_{\zeta^e}(M),\label{eq::Whfun} \\
&Q_{\zeta^e}\tens{U(\g_\ell,e)}(\_):U(\g_\ell, e)\Mod  \rightarrow \g_\ell\text{-Wmod}^{\zeta^e},~V\mapsto Q_{\zeta^e}\tens{U(\g_\ell,e)}V.
\end{align}
We refer to $\text{Wh}_{\zeta^e}(\_)$ from \eqref{eq::Whfun} as the {\em Whittaker functor} corresponding to $\zeta^e$.

The following result, proved by He   \cite[Theorem 2.4.8]{He18}, is an analogue of the  Skryabin type equivalence  \cite{Skr02} for the finite $W$-algebras associated with  Takiff algebras; see also \cite{He22}.
\begin{lem}[He] \label{thm::21}
	We have mutually inverse equivalences   $\emph{Wh}_{\zeta^e}(\_)$  and $Q_{\zeta^e}\otimes_{  W_{\zeta^e}}(\_)$ between  $\g_\ell\emph{-Wmod}^{\zeta^e}$ and  $U(\g_\ell, e)\Mod$.
\end{lem}

\label{rem::22}Let $\zeta(\_): \mf n_\ell\rightarrow \C$ be an arbitrary non-singular character. 
  Recall the one-dimensional simple $Z(\g_\ell)U(\mf n_\ell)$-module $\C_{\chi,\zeta}$ from Subsection \ref{sect::42}, where   $\chi:Z(\g_\ell)\rightarrow\C$ is a  central character. 
  We define  the following $\g_\ell$-module 
$$L(\chi,\zeta) := \text{Ind}_{{Z}(\mathfrak{g}_{\ell}) {U}(\mathfrak{n}_{\ell})}^{{U}(\mathfrak{g}_{\ell})}\C_{\chi,\zeta} = {U}(\mathfrak{g}_{\ell})\tens{Z(\g_\ell)U(\mf n_\ell)}\C_{\chi,\zeta}.$$ 

Furthermore, we define the category   $\g_\ell$-Wmod$^{\zeta}$ of $\g_\ell$-modules in the same fashion.
We may note that the equivalence $T_{\zeta,\zeta^e}$ from Lemma \ref{lem::200} extends to an equivalence between $\g_\ell\text{-Wmod}^{\zeta^e}$ and $\g_\ell\text{-Wmod}^{\zeta}$. Consequently, we have  $\g_\ell\text{-Wmod}^{\zeta}\cong U(\g_\ell, e)\Mod$.  This observation yields the following consequence, which has been obtained by Xia \cite[Theorem 5.7]{X23}; see also He \cite[Corollary 4.7]{He22}, where the special case that $\zeta=\zeta^e$ has been studied. Here, we provide an alternative proof below.

\begin{lem}[He, Xia] \label{cor::26}
	The set $$\{L(\chi,\zeta)|~\chi:Z(\g_\ell)\rightarrow \C \text{ is a character}\}$$ is a complete set of representatives of the isomorphism classes of simple objects in $\mc N_\ell(\zeta)$. 
\end{lem}
\begin{proof}
	First, we assume that $\zeta = \zeta^e$.  Theorem \ref{thm::21}, combined with the fact that   $U(\g_\ell,e)\cong Z(\g_\ell)$ (\cite[Theorem 2.4.2]{He18}), has the consequence that the following set 
	\begin{align*}
		&\{Q_{\zeta^e}\tens{Z(\g_\ell)} \C_{\chi}|~\chi: Z(\g_\ell) \rightarrow \C\text{ is a central character}.\}  
	\end{align*} 	is a complete and irredundant set of representatives of isomorphism classes of   $\mc N_\ell(\zeta^e)$.   We may observe that $L(\chi,\zeta^e)\cong Q_{\zeta^e}\tens{Z(\g_\ell)} \C_{\chi}$, for any central character $\chi$ of $\g_\ell$.  Since $T_{\zeta,\zeta^e}(L(\chi,\zeta)) \cong L(\chi,\zeta^e)$, the conclusion follows.
\end{proof}
 
 The following is a Takiff  analogue of Kostant type equivalence  \cite{Ko78}.
 \begin{prop} \label{prop::27}  Let $\zeta:\mf n_\ell\rightarrow \C$ be a non-singular character. Then 
  $\mc N_\ell(\zeta)$ is equivalent to  the category of finite-dimensional $Z(\g_\ell)$-modules.
 \end{prop}
\begin{proof}
	 By \cite[Theorem 2.4.2]{He18},  we have $Z(\g_\ell)\cong U(\g_\ell,e)$.  By Lemmas  \ref{lem::200} and \ref{lem::21}, it suffices to show that the Whittaker functor $\text{Wh}_{\zeta^e}(\_)$ restricts to an equivalence from 
		$\mc N_\ell(\zeta^e)$  to the category of finite-dimensional $U(\g_\ell,e)$-modules. Since $\mc N_\ell(\zeta^e)$ is closed under extension by Corollary \ref{cor::27}, it remains to show that every object in  $\mc N_\ell(\zeta^e)$ has finite length.  We set $\zeta := \zeta^e$.
	
	Suppose that $M\in \mc{N}_\ell(\zeta)$. Then $M$ is generated by a finite-dimensional $Z(\g_\ell)U(\mf n_\ell)$-module $V$. This implies that $M$ is a quotient of  $\text{Ind}_{{Z}(\mathfrak{g}_{\ell}) {U}(\mathfrak{n}_{\ell})}^{{U}(\mathfrak{g}_{\ell})}V$. 	We claim  that the module  $\text{Ind}_{Z(\g_\ell)U(\mf n_\ell)}^{{U}(\mathfrak{g}_{\ell})}V$ is of finite length.  We shall proceed with the proof by induction on the dimension of $V$. If $V$ is simple, then by Corollary \ref{cor::24} we have  $V\cong\mathbb{C}_{\chi,\zeta}$, for some character  
	$\chi: {Z}(\mathfrak{g}_\ell)\rightarrow\mathbb{C}$. In this case, we know that  $\text{Ind}_{Z(\g_\ell)U(\mf n_\ell)}^{{U}(\mathfrak{g}_{\ell})}V$ is simple by Lemma \ref{cor::26}. 
	
	Now, suppose that $V$ is an arbitrary finite-dimensional $Z(\g_\ell)U(\mf n_\ell)$-module.  Let $V'\subseteq V$ be a submodule of $V$ such that  $V/{V'}$ is simple. Then we get an exact sequence
	$$\text{Ind}_{Z(\g_\ell)U(\mf n_\ell)}^{{U}(\mathfrak{g}_{\ell})}V'\rightarrow \text{Ind}_{Z(\g_\ell)U(\mf n_\ell)}^{{U}(\mathfrak{g}_{\ell})}V\rightarrow \text{Ind}_{Z(\g_\ell)U(\mf n_\ell)}^{{U}(\mathfrak{g}_{\ell})}{V/V'}\rightarrow 0.$$ By induction, $\text{Ind}_{Z(\g_\ell)U(\mf n_\ell)}^{{U}(\mathfrak{g}_{\ell})}V'$ is of finite length, and thus $\text{Ind}_{Z(\g_\ell)U(\mf n_\ell)}^{{U}(\mathfrak{g}_{\ell})}V$ is of finite length. Consequently, $M$ has finite length. 
\end{proof}

\subsection{Classification of non-singular simple Whittaker  $\wtg_\ell$-modules} \label{sect::433}
A $\wtg_\ell$-module $M$ is said to be a Whittaker $\wtg_\ell$-module if it restricts to a Whittaker $\g_\ell$-module. 
For a given non-singular character $\zeta: \mf n_\ell\rightarrow\C$. We denote by $\widetilde{\mc N}_\ell(\zeta) $  the category of all    $\wtg_\ell$-modules $M$ such that $\Res M\in  {\mc N}_\ell(\zeta)$. For any $M\in\widetilde{\mc N}_\ell(\zeta) $, we define the Whittaker vector subspace corresponding to $\zeta$ as follows:
\begin{align*}
	&\text{Wh}_\zeta(M):=\{m\in M|~xm=\zeta(x)m,~\text{for any }x\in \mf n_\ell\}.
\end{align*} Recall that equivalence $T_{\zeta,\zeta^e}: \mc N_\ell(\zeta)\rightarrow \mc N_\ell(\zeta^e)$. For any $M\in \mc N_\ell(\zeta)$,  we note that the space  $\text{Wh}_{\zeta^e}(T_{\zeta,\zeta^e}(M))$ can be interpreted as the space  $\text{Wh}_\zeta(M).$
 The following lemma will be useful.

\begin{lem} \label{lem::20}
	Let $\zeta: \mf n_\ell\rightarrow\C$ be a non-singular character. 
Suppose that  $M$ is a \emph{(}not necessarily finitely generated\emph{)} $\wtg_\ell$-module on which $x-\zeta(x)$ acts locally nilpotently, for any $x\in \mf n_\ell$. Then we have $M =U(\g_\ell)\cdot  \emph{Wh}_\zeta(M).$
\end{lem}
\begin{proof}
Recall the character $\zeta^e(\_):\mf n_\ell\rightarrow\C$ from Subsection \ref{sect::43}. In the case that $\zeta =\zeta^e$, the conclusion follows by Lemma \ref{thm::21}. Next, assume that $\zeta$ is an arbitrary non-singular character. As we have already mentioned above, the functor $T_{\zeta,\zeta^e}$ from Lemma \ref{lem::200} leads to an equivalence $T_{\zeta,\zeta^e}(\_):\g_\ell\text{-Wmod}^{\zeta}$ and $\g_\ell\text{-Wmod}^{\zeta^e}$. Recall that $T(M)$ and $M$ have the same underlying space. Furthermore, the Whittaker vector subspace $\text{Wh}_{\zeta^e}(T_{\zeta,\zeta^e}(M))$ can be identified as the subspace  $\text{Wh}_\zeta(M)\subseteq T_{\zeta,\zeta^e}(M).$ Since $T(M) =U(\g_\ell)\cdot \text{Wh}_{\zeta^e}(T_{\zeta,\zeta^e}(M))$,  the conclusion follows.\end{proof}

\begin{lem}\label{lem::24} 	Let $\zeta: \mf n_\ell\rightarrow\C$ be a non-singular character. Then 
the category ${\mc N}_\ell(\zeta)$ is closed under tensoring with finite-dimensional modules, that is, suppose that $M\in {\mc N}_\ell(\zeta)$, then $V\otimes M\in {\mc N}_\ell(\zeta),$ for any  finite-dimensional $\g_\ell$-module $V$. In particular, the 
 functors $\Ind(\_)$, $K(\_)$ and $\Res(\_)$ restrict to exact functors between $\mc N_\ell(\zeta)$ and $\widetilde{\mc  N_\ell}(\zeta)$.
\end{lem}
\begin{proof} 
	It is clear that $V\otimes M\in \g_\ell\text{-Wmod}^{\zeta}$, namely,  $x-\zeta(x)$ acts on  $V\otimes M$ locally nilpotently, for any $x\in \mf n_\ell$. The fact that $V\otimes M$ is a finitely generated $\g_\ell$-module can be proved by a standard argument; see, e.g., \cite[Subsection 1.1]{Hu08}. By Lemma \ref{lem::20}, it suffices to prove that $\text{Wh}_\zeta(V\otimes M)<\infty$.  
	 By using arguments completely analogous to those in the proof of \cite[Theorem 8.1]{BK08},  there follows that  a canonical    isomorphism  	 $\text{Wh}_\zeta(V\otimes M) \cong V\otimes \text{Wh}_\zeta(M).$
	 Such an isomorphism first appeared in 1979 in the  Ph.D. thesis  of  Lynch \cite{Ly79}. The conclusion now follows from Proposition \ref{prop::27}.
\end{proof}

\begin{prop} Let $\zeta: \mf n_\ell\rightarrow\C$ be a non-singular character. Then 
every object in  $\widetilde{\mc N}_\ell(\zeta)$ has finite length.  Furthermore,    $\widetilde{\mc N}_\ell(\zeta)$ is the Serre subcategory of $\wtg_\ell\mod$ generated by  all   simple $\wtg_\ell$-modules  on which $x-\zeta(x)$ acts locally nilpotently, for any $x\in \mf n_\ell$.
\end{prop}
\begin{proof}
	The first assertion is a direct consequence of  Proposition \ref{prop::27} and Lemma \ref{lem::24}. To prove the second assertion, we let $L$ be a simple $\wtg_\ell$-module on which $x-\zeta(x)$ acts locally nilpotently, for any $x\in \mf n_\ell$. Then $L$ is a quotient of $K(V)$, for some simple object $V\in \mc N_\ell(\zeta)$ by \cite[Theorem A]{CM21}. By Lemma \ref{lem::24} we may infer that $L\in \mc N_\ell(\zeta)$. Now, the conclusion follows by Corollary \ref{cor::27}.
\end{proof}

We recall the standard Whittaker module $\widetilde{M}(\la,\zeta)  := K(L(\la,\zeta)),$ for $\la\in \h^\ast_\ell$.
  \begin{thm} \label{thm::29} Suppose that $\wtg=\gl(m|n), \mf{osp}(2|2n),$ or $\pn$. Let $\zeta :\mf n_\ell\rightarrow\C$ be a non-singular character. 
 Then  	we have 
  	\begin{itemize}
  		\item[(1)]
  		 For each central character $\chi:Z(\g_\ell)\rightarrow\C$, the standard Whittaker module $\widetilde{M}(\chi,\zeta)$ has a simple top, which we denote by  $\widetilde{L}(\chi,\zeta)$. Both modules $\widetilde{M}(\chi,\zeta)$ and $\wtL(\chi,\zeta)$ are objects in $\widetilde{\mc N}_\ell(\zeta)$. 
  		\item[(2)] $\{\widetilde{L}(\chi,\zeta)|~\chi:Z(\g_\ell)\rightarrow\C \text{ is a character}\}$ is a complete and irredundant set of representatives of isomorphism classes of simple	objects  in $\widetilde{\mc N}_\ell(\zeta)$.
  	\end{itemize}                                                                                                                                                                
  \end{thm}
\begin{proof}  Let $\la \in \h^\ast_\ell$. 
Since $\widetilde{M}(\la,\zeta) = K({L}(\la,\zeta))$, the fact that $\widetilde{M}(\la,\zeta)$ has a simple top is a consequence of \cite[Theorem A]{CM21}. Also, it follows by Lemma \ref{lem::24} that   $\widetilde{M}(\la,\zeta), \wtL(\la,\zeta) \in \widetilde{\mc N}_\ell(\zeta)$. 
  
  Next, we are going to prove Part (2). Let $S$ be a simple object in $\widetilde{\mc N}_\ell(\zeta)$. By \mbox{\cite[Theorem A]{CM21}}, $S$ is isomorphic to  the top of $K(V_S)$, for some simple $\g_\ell$-module $V_S$. Since  $V_S$ is a direct summand of $\Res S$ (in fact, $V_S\cong S^{\wtg^1_\ell}$, by  \mbox{\cite[Remark 3.3]{CM21}} and \cite[Corollary 4.3]{CM21}),  it follows by Lemma \ref{lem::24} that $V_S$ is an object in $\mc N_\ell(\zeta)$. This proves the assertion in Part (2). \end{proof}

\begin{proof}[Proof of Theorem \ref{thm::3rd}]
The first assertion follows from Theorem \ref{thm::29}. Also, Theorem \ref{thm::1st}, combined with Theorem  \ref{prop::Kacsimple} and Proposition \ref{lem::4}, gives the second assertion in \eqref{eq::9}. This completes the proof. 
\end{proof}

 \appendix{}
\section{Linkage principle for modules in the category $\wto_\ell$} \label{App::blocks}
 The goal of this appendix is to complete the proof of Proposition \ref{lem::blockdec}. Throughout this subsection, we let $\wtg$ be a basic classical Lie superalgebra with a fixed triangular decomposition $\wtg=\widetilde{\mf n}^-\oplus \mf h \oplus \widetilde{\mf n}$ from \eqref{sect::200}. We shall keep the notations
 from previous sections. The proof follows the same strategy as in the Lie algebra
 case given in \cite[\mbox{Chapter III}, Subsections 3.1, 3.3]{K02}. In the following, we will only provide the main steps, while omitting the parts that are completely analogous to the Lie algebra
 case and for which we refer to loc. cit. for details. 
 \subsection{Tor and Ext functors} \label{sect::A1}
Denote by $\wtg_\ell\Mod$ and $\text{Mod-}\wtg_\ell$ the categories of left and right $\wtg_\ell$-modules, respectively.  
 Let ${\C}$-Mod denote the category of all vector spaces over $\C$.  For any $k\geq 0$, we define the   {\em Tor} and {\em Ext} functors (e.g., see \mbox{\cite[Appendix D]{K02}}):
\begin{align}
&\Tor_k^{\wtg_\ell}(\_,\_):~  \text{Mod}\text{-}\wtg_\ell \times \wtg_\ell\Mod \rightarrow \C\text{-Mod},\\
&\Ext_{\wtg_\ell}^k(\_,\_):~\wtg_\ell\Mod \times \wtg_\ell\Mod \rightarrow \C\text{-Mod}.
\end{align}

 Let $(\_)^\omega: \wtg\rightarrow \wtg$ be the the anti-automorphism  of   \cite[Proposition 8.1.6]{Mu12}. 
We may note that it extends to an anti-automorphism of $\wtg_\ell$:
\begin{align*}
&\sum_{i=0}^{\ell}x_i\otimes \theta^i \mapsto \sum_{i=0}^\ell x_i^{\omega}\otimes \theta^i, \text{ for $ x_0,x_1\ldots,x_\ell \in \wtg$}.
\end{align*}  Abusing the notation, we shall also denote this map by  $(\_)^{\omega}$.  Let $F(\_): \wtg_\ell\Mod \rightarrow \text{Mod-}\wtg_\ell$ be the equivalence induced by $(\_)^{\omega}$, that is, for any $M\in \wtg_\ell\mod$, $F(M)$ has the same underlying subspace as $M$ with the right $\wtg_\ell$-action given by  
$mx:=(-1)^{\ov m\cdot \ov x}x^{\omega}m,$ for homogeneous elements $x\in U(\wtg_\ell)$ and $m\in M$. Here the notation $\ov v$, as usual,  denotes the parity of $v$ for a given homogeneous element $v$.  For simplicity, we define 
$\Tor_k^{\wtg_\ell}(M,N) := \Tor_k^{\wtg_\ell}(F(M),N)$, for any   left $\wtg_\ell$-module $M$ and right $\wtg_\ell$-module $N$.

For  $M\in \wtg_\ell\Mod$, we let $M^\ast : =\Hom_\C(M,\C)$ denote the full dual representation of $M$. In addition, 
 $M$ is said to be a $(\wtg_\ell,\mf h)$-module if it is locally finite and semisimple over $\h$.  In this case, we let $M^\vee$ denote the {\em restricted dual} of $M^\ast$, that is, $M^\vee  :=\bigoplus_{\nu\in {\h^\ast}} (M^{\nu})^\ast \subseteq M^\ast$ is a $\wtg_\ell$-submodule of $M^\ast$ consisting of $\h$-semisimple part of $M^\ast$. Here $M^\nu$ denotes the weight subspace of $M$ associated to $\nu$; see Subsection \ref{sect::222}. Furthermore, we define  a new $\wtg_\ell$-module structure of the restricted dual of $M$ by declaring that $(x\cdot f)m = (-1)^{\ov f\ov x}f(x^{\omega}m),$ where  $f\in M^\vee$, 
$x\in \wtg_\ell$  are homogenous and $m\in M$.  With this new $\wtg_\ell$-action we denote $M^\vee$ by $M^\sigma$. Then this gives rise to an exact functor $(\_)^\sigma: \wto_\ell\rightarrow\wtg_\ell\Mod$. Similarly, we let $(\_)^\tau: \wto_\ell\rightarrow\wtg_\ell\Mod$ denote the exact functor induced by the anti-automorphism $(\_)^{{\omega}^{-1}}$ of $\wtg_\ell$. Then we have $(M^\sigma)^\tau\cong M \cong (M^\tau)^\sigma$, for any $M\in \wto_\ell$. 

\subsection{Basic properties}
In this subsection, we prepare some useful tools. For any $\wtg_\ell$-module $M$, we let $M^\omega$ denote the twisted $\wtg_\ell$-module with the same underlying space as $M$ and the new $\wtg_\ell$-action twisted by the automorphism $-(\_)^\omega$. Furthermore, let $M^{\omega'}$ denote the twisted $\wtg_\ell$-module by the auto-morphism of $-(\_)^\omega$.  
 \begin{lem} \label{lem::9}  Let $M,N$ be two  $\wtg_\ell$-modules. Then  we have 
 	\begin{align}
 	&\emph{\Tor}_k^{\wtg_\ell}(M,N)\cong  \emph{\Tor}_k^{\wtg_\ell}(\C, M\otimes N)\cong
 	 \emph{\Tor}_k^{\wtg_\ell}(N,M), \label{eq::144}\\
 	 	&\emph{\Tor}_k^{\wtg_\ell}(M^{\omega'},N)\cong 
 	 \emph{\Tor}_k^{\wtg_\ell}(M,N^{\omega}),\label{eq::155}\\
 	&\Ext_{\wtg_\ell}^k(M,N^\ast)\cong \emph{\Tor}_k^{\wtg_\ell}(M,N)^\ast, \label{eq::166}
 	\end{align} for any $k\geq 0.$
 \end{lem}
\begin{proof} The proof of \cite[Lemma 3.1.13]{K02} can be adapted to establish the isomorphisms in \eqref{eq::144},  \eqref{eq::166}, and thus we omit the proof.
	
	To prove the isomorphism in  \eqref{eq::144}, we let $H_k(\wtg_\ell, V)$ be the $k$-th Lie superalgebra homology of $\wtg_\ell$ with coefficients in $V$. 
	It is standard to interpret $H_k(\wtg_\ell, V)$ as $\Tor_k^{\wtg_\ell}(\C,V)$; see also \cite[Lemma 3.1.9]{K02}.
 	It follows by  an argument completely analogous to the proof of \cite[Lemma 3.3.3]{K02} that $H_k(\wtg_\ell, V) \cong  H_k(\wtg_\ell, V^{\omega})$. Together with \eqref{eq::155}, \eqref{eq::166}, this establishes the isomorphism in \eqref{eq::166} follows. 
\end{proof}

 The following lemma is an analogue of \cite[Lemma 3.3.5]{K02}. 
\begin{lem}\label{lem::10} Suppose that $M,N$ are $(\wtg_\ell,\mf h)$-modules. Then 
	\begin{align}
	&\Ext^k_{\wtg_\ell}(M,N^\vee) \cong \Ext_{\wtg_\ell}^k(M,N^\ast), \text{ for any $k\geq 0$.} \label{eq::14}
	\end{align}
Therefore, we have 
\begin{align}
&\Ext_{\wtg_\ell}^k(M,N)\cong {\emph{\Tor}}^{\wtg_\ell}_k(M^{\omega },N^\sigma)^\ast\cong {\emph{\Tor}}^{\wtg_\ell}_k(M^{\omega' },N^\tau)^\ast, \text{ for any $k\geq 0$.} \label{eq::15}
\end{align} 
\end{lem}
\begin{proof}
	Since the adjoint module of $\wtg_\ell$ is a $(\wtg_\ell,\h)$-module, 	the proof of \mbox{\cite[Lemma 3.3.5]{K02}} can be adapted in a straightforward way to prove the isomorphism in  \eqref{eq::14}.
 To establish the isomorphism in \eqref{eq::15}, we calculate
\begin{align*}
&\Ext_{\wtg_\ell}^k(M,N)\\
&\cong \Ext_{\wtg_\ell}^k(M,(N^{\sigma})^\tau)\\
&\cong \Ext_{\wtg_\ell}^k(M,((N^{\sigma})^{\omega'})^\ast) \text{ by \eqref{eq::14}}\\
&\cong \Tor^{\wtg_\ell}_k(M,(N^{\sigma})^{\omega'})^\ast \text{ by \eqref{eq::166}}\\
&\cong \Tor^{\wtg_\ell}_k(M^{\omega},N^{\sigma})^\ast \text{ by \eqref{eq::155}.}
\end{align*} Similarly for the isomorphism $\Ext_{\wtg_\ell}^k(M,N)\cong {\text{\Tor}}^{\wtg_\ell}_k(M^{\omega' },N^\tau)^\ast$. This completes the proof.
\end{proof}

The following lemma is an analogue of \cite[Lemma 3.1.14]{K02}. 
\begin{lem} \label{lem::11} 
	Let $\mf s \supseteq \mf t$ be Lie superalgebras. Then for any $\mf s$-module $M$ and $\mf t$-module $N$, we have ${\emph{\Tor}}_k^{\mf s}(M,U(\mf s)\otimes_{\mf t} N)\cong {\emph{\Tor}}_k^{\mf t}(\Res_{\mf t}^{\mf s}M, N), \text{ for any $k\geq 0$.}$
	\end{lem}
  \begin{proof}
Adapt the proof of  \cite[Lemma 3.1.14]{K02}.
  \end{proof}

\subsection{Extension groups and the proof of Proposition \ref{lem::blockdec}} 
Recall the equivalence relation $\sim$ on $\h^\ast_\ell$ introduced in Subsection \ref{sect::223}. 
Let $\Lambda\subseteq \mf h^\ast_\ell/\sim$ be an equivalence class. 

A $\wtg_\ell$-module $M$ is said to be of {\em type $\Lambda$} provided that  any  simple subquotient of $M$ is isomorphic to a simple highest weight module $\wtL(\la)$ such that  $\la\in \Lambda.$  Here we extend the same terminology, as introduced for $\wto_\ell$ in Subsection \ref{sect::223}, for $\wtg_\ell\Mod$.

\begin{prop}  \label{prop::13}
	Suppose that  $\la,\mu\in \h^\ast_\ell$ with $\la\neq \mu$. Then  
	\begin{align} 
	&\Ext_{\wtg_\ell}^1(\widetilde{M}(\la),\widetilde{M}(\mu)^{\sigma})=\Ext_{\wtg_\ell}^1(\widetilde{M}(\la),\widetilde{M}(\mu)^{\tau}) =0. \label{eq::vani1}
	\end{align} 
Furthermore, suppose that $\la\not\sim \mu$. Then  for any quotient  $\widetilde{M}(\la)\twoheadrightarrow  M$ and submodule  $N\hookrightarrow \widetilde{M}(\mu)^\sigma$ we have 
\begin{align}
&\Ext^1_{\wtg_\ell}(M,N) =0. \label{eq::21}
\end{align}
\end{prop}
\begin{proof}
   By \eqref{eq::15}, it suffices to show that 
   \begin{align}
    &{\Tor}^{\wtg_\ell}_1(\widetilde{M}(\la)^\omega,\widetilde{M}(\mu))=0. \label{eq::vani2}
   \end{align}
 
 We will adapt the proof of  \mbox{\cite[Theorem 3.3.8]{K02}}. We calculate
\begin{align*}
	&{{\Tor}}^{\wtg_\ell}_1(\widetilde{M}(\la)^\omega,\widetilde{M}(\mu)) \\
	&\cong {{\Tor}}^{\widetilde{\mf b}_\ell}_1(\Res_{\widetilde{\mf b}_\ell}^{\wtg_\ell}\widetilde{M}(\la)^\omega,\C_{\mu}) \text{ by Lemma \ref{lem::11}} \\
	&\cong {{\Tor}}^{\widetilde{\mf b}_\ell}_1(\C_{\mu}, \Res_{\widetilde{\mf b}_\ell}^{\wtg_\ell}\widetilde{M}(\la)^\omega) \text{ by Lemma \ref{lem::9}} \\
	&\cong {{\Tor}}^{{\mf h}_\ell}_1(\C_{\mu}, \C_{-\la} )  \text{ by Lemma \ref{lem::11} and $\Res_{\widetilde{\mf b}_\ell}^{\wtg_\ell}\widetilde{M}(\la)^\omega\cong U(\widetilde{\mf b}_\ell)\otimes_{{\mf h}_\ell}\C_{-\la}$}.
\end{align*}
We claim that ${{\Tor}}^{\widetilde{\mf h}_\ell}_1(\C_{\mu}, \C_{-\la}) = {{\Tor}}^{\widetilde{\mf h}_\ell}_1(\C, \C_{\mu-\la})=0$. To prove this, we let $H_1({\mf h}_\ell, \C_{\mu-\la})$  denote the $1$-st Lie algebra homology of $\h_\ell$ with coefficients in $\C_{\mu-\la}$. By a standard argument we have ${{\Tor}}^{\widetilde{\mf h}_\ell}_1(\C, \C_{\mu-\la})\cong H_1({\mf h}_\ell, \C_{\mu-\la})$, which is a trivial ${\h}_\ell$-module (e.g., see  \cite[Subsection 3.1.1]{K02}). This completes the proof of  \eqref{eq::vani2}.

To prove \eqref{eq::21}, let 
\begin{align*}
&0\rightarrow X\rightarrow \widetilde{M}(\la) \rightarrow M\rightarrow 0,\\
&0\rightarrow N\rightarrow \widetilde{M}(\mu)^\sigma\rightarrow Y\rightarrow 0,
\end{align*} be two short exact sequences in $\wtg_\ell\Mod$. Applying the bi-functor $\Ext^{1}_{\wtg_\ell}(\_,\_)$, we obtain the following commutative diagram with exact rows and columns 
	\begin{align*}
	&\xymatrixcolsep{2pc} \xymatrix{
		& \Hom_{\wtg_\ell}(X,N) \ar[d] &  \Hom_{\wtg_\ell}(X,\widetilde{M}(\mu)^\sigma) \ar[d]  \\
		\Hom_{\wtg_\ell}(M,Y) \ar[r]     & \Ext^1_{\wtg_\ell}(M,N) \ar[r] \ar@<-2pt>[d]   & \Ext^1_{\wtg_\ell}(M,\widetilde{M}(\mu)^\sigma) \ar[d] \\ \Hom_{\wtg_\ell}(\widetilde M(\la),Y) \ar[r]    &  \Ext^1_{\wtg_\ell}(\widetilde{M}(\la),N)  \ar[r] &\Ext^1_{\wtg_\ell}(\widetilde{M}(\la),\widetilde{M}(\mu)^\sigma).}  
\end{align*} Since $\la$ and $\mu$ are of different types, it follows that  $$\Hom_{\wtg_\ell}(X,N)=\Hom_{\wtg_\ell}(X,\widetilde{M}(\mu)^\sigma) = \Hom_{\wtg_\ell}(M,Y)=\Hom_{\wtg_\ell}(\widetilde M(\la),Y)=0.$$ Consequently, the vanishing for  $\Ext_{\wtg_\ell}^1(M,N)$ follows by \eqref{eq::vani1}.
\end{proof}

\begin{cor} \label{coro::14} Suppose that $\Lambda_1,\Lambda_2\in  \h^\ast_\ell/\sim $ are two equivalence classes. Let $\la\in \Lambda_1$ and $M$ be a quotient of  $\widetilde{M}(\la)$.  Let $N$ be a  finite-length $\wtg_\ell$-module of  type $\Lambda_2$. Then we have 
	\begin{align}
		&\Ext^1_{\wtg_\ell}(M,N)\neq 0 \Rightarrow \Lambda_1=\Lambda_2.
	\end{align}

 In particular, we have   
\begin{align}
	&\Ext^1_{\wtg_\ell}(\wtL(\la),\wtL(\mu))\neq 0 \Rightarrow  \Lambda_1=\Lambda_2,\label{eq::22}
\end{align} for any $\mu\in \Lambda_2$.
\end{cor}
\begin{proof} Using induction on the length of $N$, the conclusion follows by Proposition \ref{prop::13}. \end{proof}

We are now in a position to give a proof of Proposition \ref{lem::blockdec} 
 
 \begin{proof} [Proof of Proposition \ref{lem::blockdec}]We shall adapt the arguments used in \cite[Corollary 3.3.10]{K02}. Let $M$ and $N$ be two objects in $\wto_\ell$ of  types $\Lambda_1$ and $\Lambda_2$, respectively, for some  different  equivalence classes $\Lambda_1\neq \Lambda_2\subseteq \h^\ast_\ell/\sim$. We are going to show that  $\Ext_{\wtg_\ell}^1(M,N)=0$.  
     Thanks to Lemma \ref{lem::1}-(iii), without loss of generality, we may assume that $M$ and $N$ are highest weight modules.  By Lemma \ref{lem::10},  it suffices to prove that  $  {{\Tor}}^{\wtg_\ell}_1(M^\omega,N^\sigma)=0.$ By Proposition \ref{prop::12}, there exists a filtration $N=N_0\supseteq  N_1\supseteq N_2\supseteq \cdots,$ such that $\bigcap_{i\geq 0}N_i =0$ and  $N_j/N_{j+1}$ are simple, for all $j\geq 0$. 
Therefore, $N^\sigma$ admits a filtration  
 $0=X_0 \subseteq X_1 \subseteq X_2  \subseteq \cdots $  
 such that each  $X_{j+1}/X_j\cong (N_{j}/N_{j+1})^\sigma\cong N_{j}/N_{j+1}$ is  simple. By considering weight spaces, it follows that the union of all $X_i$ is $N^{\sigma}$. Since the Tor functor commutes with direct limits, the conclusion about the vanishing of $\Ext_{\wtg_\ell}^1(M,N)$ follows by \mbox{Corollary \ref{coro::14}}. Finally, we note that every non-zero object in $\wto_\ell$ contains a highest weight module by Lemma \ref{lem::1}. Using arguments similar to those in the proof of  \cite[Corollary 3.3.10]{K02}, together with the results proved above, the proof of Proposition \ref{lem::blockdec} is completed.
 \end{proof}

The authors have no conflicts of interest to declare that are relevant to this article.


\begin{thebibliography}{9999999} 
	
	
	\bibitem[AP17]{AP17}	T. Arakawa and A. Premet. {\em Quantizing Mishchenko–Fomenko subalgebras for centralizers
	via affine $W$-algebras}. Trans. Mosc. Math. Soc. {\bf78} (2017), 217--234.
	
	\bibitem[AB21]{AB21}
	J. C. Arias, E. Backelin. 
	{\em Projective and Whittaker functors on category $\mathcal O$}. J. Algebra {\bf 574}
	(2021), 154--171.

	\bibitem[BR13]{BR13}
	A.~Babichenko and D.~Ridout. 
	{\em Takiff superalgebras and conformal field theory.} J. Phys. A {\bf 46} (2013), no. 12, 26pp.
	
		\bibitem[BC15]{BC15}
A.~Babichenko, and T.~Creutzig. {\em Harmonic analysis and free field realization of the Takiff supergroup of $GL(1|1)$.} SIGMA Symmetry Integrability Geom.  Methods Appl. {\bf 11} (2015), Paper 067, 24pp.

\bibitem[Ba97]{B97} E.~Backelin. 		{\em Representation of the category $\mathcal O$ in Whittaker categories}.  		Int. Math. Res. Not. IMRN  		{\bf 4} (1997), 153--172.

\bibitem[BCW14]{BCW14}
	I. Bagci, K. Christodoulopoulou, E. Wiesner. {\em Whittaker
		Categories and Whittaker modules for Lie
		superalgebras}. Comm. Algebra {\bf 42} (2014), no. 11, 4932--4947.


\bibitem[B+9]{B+9}
M.~Balagovic, Z.~Daugherty, I. Entova-Aizenbud, I. Halacheva, J. Hennig, M.-S. Im, G. Letzter, E. Norton, V. Serganova, C. Stroppel:  {\em Translation functors and decomposition numbers for the
periplectic Lie superalgebra $\mf p(n)$}. Math. Res. Lett. {\bf 26}(3), 643--710 (2019)
 


\bibitem[BaW18]{BW18} H.~Bao and W.~Wang, {\em A new approach to Kazhdan-Lusztig theory of type B via quantum symmetric pairs}.   Ast\'erisque {\bf 402}, 2018, vii+134pp.
		
\bibitem[BM11]{BM11}	P. ~Batra and V.~Mazorchuk. 		{\em Blocks and modules for Whittaker pairs.} 			J. Pure Appl. Algebra {\bf 215} (2011), 1552--1568.

	
	\bibitem[BF93]{BF93}
	A. Bell, R. Farnsteiner. 
	{\em On the theory of Frobenius extensions and its application to
	Lie superalgebras.} Trans. Amer. Math. Soc. {\bf335} (1993), no. 1, 407--424.
		
		\bibitem[BR20]{BR20}
		A.~Brown and A.~Romanov.
		{\em Contravariant forms on Whittaker modules}.
		Proc. Amer. Math. Soc. {\bf 149} (2021), 37--52.



  \bibitem[Br03]{Br03}
J.~Brundan. {\em Kazhdan--Lusztig polynomials and character formulae for the Lie superalgebra $\mf{gl}(m|n)$}. J.
Am. Math. Soc. 16, 185--231 (2003).

			\bibitem[BGK08]{BGK08}
		J. Brundan,~S. Goodwin and A. Kleshchev.
		{\em Highest weight theory
			for finite $W$-algebras}. Int. Math. Res. Not. IMRN {\bf 15} (2008), no. 15, 53pp.
		
			\bibitem[BK08]{BK08}
			J. Brundan,~A. Kleshchev.
 {\em Representations of shifted Yangians and finite $W$-algebras.}   Mem. Amer.
 Math. Soc. {\bf 196} (2008), no. 918, viii+107pp.
	
			\bibitem[Cas11]{Cas11}
	P. Casati. 
	{\em Drinfeld-Sokolov hierarchies on truncated current Lie algebras.}  Banach Center Publ. {\bf 94} (2011), 163--171. 
	
	
	
		\bibitem[CM19]{CM19}
		L. Calixto, T. Macedo.
		{\em Irreducible modules for equivariant map superalgebras and their extensions}. J. Algebra {\bf517} (2019), 365--395.
		
		\bibitem[CM23]{CM23}  
		L.~Calixto and T. Macedo. 
		{\em Finite-dimensional representations of map superalgebras.} Linear Algebra Appl. {\bf676} (2023), 104--130.
	
	
		
	\bibitem[CO06]{CO06}
			P.~Casati and G.~Ortenzi. 
		{\em New integrable hierarchies from vertex operator representations of polynomial Lie algebras}. J. Geom. Phys. {\bf56} (2006), 418--449.
		
			
		\bibitem[Ch23]{Ch23}
	M.~Chaffe. 
	{\em Category $\mc O$ for Takiff Lie algebras.} Math. Z. {\bf 304} (2023), no. 1, Paper No. 14, 35pp.
	
	\bibitem[ChT23]{CT23}
	M.~Chaffe, L.~Topley. 
	{\em Category $\mc O$ for truncated current Lie algebras.} Canad. J. Math. (2023), 1--27. doi:10.4153/S0008414X23000664.

	\bibitem[C15]{C15}		C.-W. ~Chen. 					{\em  
	 Finite-dimensional representations of periplectic Lie superalgebras}. J. Algebra {\bf 443} (2015), 99--125.
	
			\bibitem[C21]{C21a}		C.-W. ~Chen. 					{\em Whittaker modules for classical Lie superalgebras}. Comm. Math. Phys.  {\bf 388} (2021), 351--383. 
						
			\bibitem[CC23]{CC23}
		  C.-W. Chen and S.-J. Cheng. {\em Whittaker categories of quasi-reductive Lie superalgebras and principal finite W-superalgebras}. Transform.~Groups, to appear  \href{https://arxiv.org/abs/2305.05550}{https://arxiv.org/abs/2305.05550}.
				
				
			\bibitem[CC24]{CC22}
			C.-W. Chen and S.-J. Cheng. {\em Whittaker categories of quasi-reductive Lie superalgebras and quantum symmetric pairs}.  Forum Math. Sigma {\bf 12} (2024), e37, 1--29. 
			
			
						
			\bibitem[CCC21]{CCC21}
	C.-W.~Chen, S.-J.~Cheng, K.~Coulembier. {\em Tilting modules for classical Lie superalgebras.} J. Lond.
	Math. Soc. (2) {\bf103} (2021), 870--900.
	

				
			\bibitem[CCM23]{CCM23}
			C.-W. Chen, S.-J. Cheng and V. Mazorchuk. {\em Whittaker categories, properly stratified categories
				and Fock space categorification for Lie superalgebras.}   Comm. Math. Phys. {\bf401}
			(2023),  717--768.


\bibitem[CCS24]{CCS24}
			C.-W. Chen, S.-J. Cheng and U. R. Suh. 
            {\em Whittaker modules of central extensions of Takiff superalgebras and finite supersymmetric W-algebras}, Preprint  \href{https://arxiv.org/abs/2410.22819}{https://arxiv.org/abs/2410.22819}.
			
				
	\bibitem[CM21]{CM21}
			C.-W. Chen and V.~Mazorchuk.
			{\em Simple supermodules over Lie superalgebras}.
			Trans. Amer. Math. Soc. {\bf 374} (2021), 899--921.
            
\bibitem[CP24]{CP24}
			C.-W. Chen,Y.-N. Peng, {\em Parabolic category for periplectic Lie superalgebras}, Transform. Groups  29, 17–45 (2024).  

			
		\bibitem[CC22]{ChCo22}
	S.-J.~Cheng and  K.~Coulembier {\em Representation theory of a semisimple extension of the Takiff superalgebra}. Int. Math. Res. Not, IMRN {\bf 18} (2022), 14454--14495.

\bibitem[CL10]{CL10} S.-J. Cheng and N.~Lam. {\em Irreducible characters of general linear superalgebra and
super duality.} Comm. Math. Phys. {\bf 298} (2010), 645--672.

	\bibitem[CLW11]{CLW11}
	S.-J. Cheng, N.~Lam, W.~Wang. {\em Super duality and irreducible characters of ortho-symplectic Lie
	superalgebras.} Invent. Math. {\bf 183} (2011) 189--224.
 
 
	\bibitem[CW12]{CW12}
	S.-J.~Cheng and W.~Wang.
	{\em Dualities and representations of Lie superalgebras}. Grad. Stud. Math. {\bf144}
American Mathematical Society, Providence, RI, (2012), xviii+302 pp.
	

		\bibitem[CoM15]{CoM15}
	K. Coulembier and V. Mazorchuk.
	{\em Extension fullness of the categories of Gelfand--Zeitlin and Whittaker modules.}
	SIGMA 	Symmetry Integrability Geom. Methods Appl. {\bf  11} (2015), Paper No. 016.

	
	\bibitem[Di96]{Di96} J.~Dixmier. {\em Enveloping algebras}.
Graduate Studies in Mathematics, {\bf 11}
American Mathematical Society, Providence, RI, (1996), xx+379 pp.

	\bibitem[EK05]{EK05} A.G.~Elashvili and V.G.~Kac.  {\em Classification of Good Gradings of Simple Lie Algebras}. Lie Groups and Invariant Theory. Amer.~Math.~Soc.~Transl.~(2) Vol.~{\bf 213} (2005) 84--104.
	
	\bibitem[GG02]{GG02} W.L. Gan and V. Ginzburg.
	{\em Quantization of Slodowy slices}. IMRN, {\bf5} (2002), 243--255.


\bibitem[G95]{G95}
	F.~Geoffriau. {\em Harish-Chandra Homomorphism for Generalized Takiff Algebras.} J. Algebra {\bf171} (1995), 444--456.
	
	
		\bibitem[Go00]{Go00}
	 M. Gorelik. {\em On the ghost centre of Lie superalgebra.} Ann. Inst. Fourier (Grenoble) {\bf50}
	(2000), no. 6, 1745--1764.
	
		\bibitem[Go02]{Go02}
	M. Gorelik.
	{\em Strongly typical representations of the basic classical Lie superalgebras}.
	J. Amer. Math.
	Soc. {\bf 15} (2002), no. 1, 167--184.

	\bibitem[Go04]{Go04}
    M. Gorelik. 
    {\em The Kac construction of the centre of U(g) for Lie superalgebras}, J. Nonlinear
Math. Phys. {\bf 11} (2004), 325--349.
	
	
	\bibitem[GrM17]{GM17}
	J.~Greenstein,  V. Mazorchuk. 
	{\em Koszul duality for semidirect products and generalized Takiff algebras.} Algebr. Represent. Theory {\bf20} (2017), 675--694.
	
		\bibitem[He22]{He22}
  X.~He. {\em Finite W-algebras associated to truncated current Lie algebras.} Glas. Mat. Ser. III {\bf 57} (2022), 17--33.
	
			\bibitem[He18]{He18}
  X.~He. {\em W-algebras associated to truncated current Lie algebras}. Doctoral dissertation (2018), Université Laval.
	
		\bibitem[Hu08]{Hu08}
	J.~Humphreys.
	{\em Representations of semisimple Lie algebras in the BGG category $\mathcal{O}$}.
	Grad. Stud. Math. {\bf94} American Mathematical Society, Providence, RI, (2008), xvi+289 pp.	
		
		\bibitem[Ka77]{Ka1}
	V.~Kac.
	{\em Lie superalgebras}.
	Adv. Math. {\bf 16} (1977), 8--96.

 \bibitem[Ka78]{Ka78}
	V.~Kac. {\it Representations of classical Lie superalgebras. In Differential Geometrical Methods in Mathematical Physics, II (Proc. Conf., Univ. Bonn, Bonn, 1977)}, Volume 676 of Lecture Notes in Math.,
	pp. 597--626. Springer, Berlin (1978).

 \bibitem[Ka84]{Ka84}
   V.~Kac.  {\em Laplace operators of infinite-dimensional Lie algebras and theta functions},
Proc. Nat. Acad. Sci. U.S.A. 81 (1984), no. 2, Phys. Sci., 645--647. MR735060 (85j:17025) xviii, 282
	
		\bibitem[Ku02]{K02}
	S.~Kumar. 
	Kac-Moody groups, their flag varieties and representation theory. Progr. Math. {\bf204}
Birkhäuser Boston, Inc., Boston, MA, (2002), xvi+606 pp.

		\bibitem[Ko75]{Ko75}
	 B. Kostant.
	  {\em On the tensor product of a finite and an infinite dimensional representation.}
	J. Fund. Anal. {\bf20} (1975), 257--285.
	
		\bibitem[Ko78]{Ko78}
	B.~Kostant.
	{\em On Whittaker vectors and representation theory}.
	Invent. Math. {\bf 48.2} (1978), 101--184.

    	\bibitem[LSS86]{LSS86}
D. Leites, M. Saveliev, and V. Serganova, 
{\em Embedding of $\mf{osp}(N/2)$ and the associated nonlinear
supersymmetric equations. Group theoretical methods in physics}, Vol. I (Yurmala, 1985), 255--297, VNU Sci. Press, Utrecht, 1986.
	
	\bibitem[Lo10a]{Lo10}
I. Losev.
{\em Quantized symplectic actions and $W$-algebras}. J. Amer. Math. Soc. {\bf 23} (2010),
35--59.


	\bibitem[Lo10b]{Lo10b} 
	I. Losev.
  {\em Finite W-algebras}. Hindustan Book Agency, New Delhi; distributed by , (2010), 1281--1307.
  
\bibitem[Lo11]{Lo11}
I. Losev.
{\em Finite dimensional representations of $W$-algebras}. Duke Math. J. {\bf159} (2011), 99--143.


\bibitem[Lo12]{Lo12}
I. Losev.
{\em On the structure of the category $\mc O$ for $W$-algebras}. Semin. Congr.  {\bf 24} (2012), 351--368.

	\bibitem[Ly79]{Ly79}
T. E. Lynch. 
{\em Generalized Whittaker vectors and representation theory}. PhD thesis,
M.I.T., 1979.





\bibitem[MacS19]{MacS19}  
	 T. Macedo and A. Savage. 
	 {\em Invariant polynomials on truncated multicurrent algebras}.  J. Pure Appl. Algebra {\bf 223} (2019), no. 1, 349--368.
	
	\bibitem[Ma14]{Ma14}
	V.~Mazorchuk.
	{\em Parabolic category $\mathcal O$ for classical Lie superalgebras}.
	Springer INdAM Ser., {\bf7} 
Springer, Cham, (2014), 149--166.	
		
		\bibitem[MM22]{MM22}
	V. Mazorchuk, R. Mrden
	{\em Lie algebra modules which are locally finite and with finite simplicities over the semisimple part}.
	Nagoya Math. J. {\bf246} (2022), 430--470.
	
	\bibitem[MS19]{MS19}
	V.~Mazorchuk and C.~S\"oderberg. {\em Category $\mc O$ for Takiff $\mf{sl}(2)$.} J. Math. Phys. {\bf60}v(2019), no. 11, 15pp.
	

	
		
	
	
	\bibitem[Mc85]{Mc85}	 E. McDowell. 	 {\em On modules induced from Whittaker modules}. 	 J. Algebra {\bf 96} (1985), 161--177.
	
	\bibitem[Mc93]{Mc93}	 E. McDowell.	  {\em  A module induced from a Whittaker module}. 	 Proc. Amer. Math.  Soc. {\bf118} (1993), no. 2, 349--354.
	
	 
	 
	
	\bibitem[MiSo97]{MiSo97} D. Mili{\v{c}}i{\'c} and W. Soergel. 	{\em The composition series of modules induced from Whittaker modules}. Comment. Math. Helv. {\bf 72} (1997), no.4, 503--520.
	
	
	\bibitem[Mo21]{M21}
	A. I. Molev. {\em Casimir elements and Sugawara operators for Takiff algebras.} J. Math. Phys. {\bf62} (2021), no.1, 12pp.
		
	
	\bibitem[MY16]{MY16} A. Moreau and R. Yu.  
	{\em Jet schemes of the closure of nilpotent orbits}. Pacific
	J. Math. {\bf281} (2016), 137--183.
	
	
	\bibitem[Mu12]{Mu12}
	I.~M.~Musson,
	Lie superalgebras and enveloping algebras.
	Grad. Stud. Math. {\bf131} American Mathematical Society, Providence, RI, (2012), xx+488 pp.

	\bibitem[PY20]{PY20}
	D. Panyushev and O. Yakimova. 
	{\em Takiff algebras with polynomial rings of symmetric invariants.} Transform. Groups {\bf25} (2020), no. 2, 609--624.

	
		\bibitem[Pr02]{Pr02}
	A. Premet.
	{\em Special transverse slices and their enveloping algebras}.
	Adv. Math. {\bf 170} (2002), 1--55.
	
	\bibitem[Pr07a]{Pr07}	A. Premet. 	{\em Enveloping algebras of Slodowy slices and the Joseph ideal}. 	J. Eur. Math. Soc. {\bf9} (2007), no. 3, 487--543.
	
	\bibitem[Pr07b]{Pr07Mo}
	A. Premet.
	{\em Primitive ideals, non-restricted representations and finite W-algebras}. Moscow Math.
	J. {\bf7} (2007), 743--762.
	

\bibitem[PS89]{PS89}
I. Penkov,~V. Serganova, 
{\em Cohomology of $G/P$ for classical complex Lie supergroups Gand characters of some atypical $G$-modules}, Ann. Inst. Fourier 39 (1989) 845--873.


			\bibitem[Q20]{Q20}
T.~Quella. 
	{\em On conformal field theories based on Takiff superalgebras.} J. Phys. Commun. {\bf4} (2020), 075013.
	
		
	\bibitem[RT92]{RT92} 
	{M. Ra{\"i}s, P. Tauvel}.
	{\em Indice et polyn{\^o}mes invariants pour certaines algèbres de Lie}.  J. Reine Angew. Math. {\bf 425} (1992), 123--140.
	

		\bibitem[Sa14]{Sa14} 
	A.~Savage. 
	{\em Equivariant map superalgebras}. 
	Math. Z.  {\bf277} (2014, 373--399.
	
		\bibitem[Se02]{Se02} 	V.~Serganova.	{\em On representations of the Lie superalgebra $\mf p(n)$}.	J. Algebra 	{\bf 258} (2002), 615--630.
		
			\bibitem[Se11]{Se11} 	V.~Serganova.
	{\em Quasireductive supergroups}.
	Contemp. Math. {\bf544}
American Mathematical Society, Providence, RI, (2011), 141--159.


		\bibitem[Ser99]{Ser99} 
        A.~Sergeev. {\em The invariant polynomials on simple Lie superalgebras.}  Represent.
Theory 3.10 (1999): 250--280.
	
		\bibitem[Sev00]{Sev00} 
	A. Sevostyanov. 
	{\em Quantum deformation of Whittaker modules and the Toda lattice}. Duke Math. J. {\bf105} 	(2000), no. 2, 211--238.
	
	\bibitem[Sk02]{Skr02}
	S. Skryabin. {\em A category equivalence}, Appendix to \cite{Pr02}.
	
	\bibitem[T71]{T71}
	S.~Takiff. 
	{\em Rings of invariant polynomials for a class of Lie algebras.} Trans. Amer. Math. Soc. {\bf160} (1971), 249--262.
		

	\bibitem[Wa11]{Wa11}
 W. Wang. 
 {\em Nilpotent orbits and finite W-algebras}.  Fields Inst. Commun. {\bf59}
American Mathematical Society, Providence, RI, (2011), 71--105.

\bibitem[Wi11]{Wi11}
B.~Wilson. 
{\em Highest-weight theory for truncated current lie algebras.} J. Algebra {\bf336} (2011),
1--27.


	\bibitem[X23]{X23}
L.~Xia. 
{\em Whittaker modules and hyperbolic Toda lattices.}  Preprint.  \href{https://arxiv.org/abs/2401.00680}{https://arxiv.org/abs/2401.00680}.

	\bibitem[Z24a]{Z24}
 X. Zhu. 
 {\em Simple modules over the Takiff Lie algebra for $\mf{sl}(2)$}. J. Math. Phys. {\bf65} (2024), 011701.
\end{thebibliography}
\end{document}